\documentclass[a4paper,12pt]{amsart}
\usepackage[english]{babel}
\usepackage{amsmath,amsthm,amssymb,amsfonts}
\usepackage{multicol}
\usepackage{todonotes}
 \usepackage{enumitem} 
\usepackage[bookmarks=true,hyperindex,pdftex,colorlinks,citecolor=blue]{hyperref}
\hypersetup{colorlinks=true,  linkcolor=blue, citecolor=red, urlcolor=cyan}

\usepackage{xcolor}
\usepackage[export]{adjustbox}
\usepackage{graphicx}
\usepackage{subcaption}

\usepackage[english]{babel}

\usepackage{graphicx}

\usepackage{tikz}
\usetikzlibrary{mindmap,backgrounds, calc}
\usetikzlibrary{matrix,chains,scopes,positioning,arrows,fit}
\usetikzlibrary{positioning, shapes, patterns}
\usetikzlibrary{automata} 
\usetikzlibrary{shapes.geometric, arrows, arrows.meta}
\usetikzlibrary{positioning,decorations.markings}

\theoremstyle{definition}
\newcounter{maincoro}

\newtheorem{theorem}{Theorem}[section]
\newtheorem{lemma}[theorem]{Lemma}
\newtheorem{fact}[theorem]{Fact}

\newtheorem{proposition}[theorem]{Proposition}
\newtheorem{corollary}[theorem]{Corollary}

\theoremstyle{definition}
\newcounter{maintheorem}

\newtheorem{definition}[theorem]{Definition}

\theoremstyle{remark}
\newtheorem{remark}[theorem]{Remark}
\numberwithin{equation}{section}

\newcommand{\R}{\mathbb{R}}
\newcommand{\C}{\mathbb{C}}
\newcommand{\N}{\mathbb{N}}

\newcommand{\T}{\mathbb{T}}
\newcommand{\D}{\mathbb{D}}

 \makeatletter
\renewcommand{\tocsection}[3]{%
	\indentlabel{\@ifnotempty{#2}{\bfseries\ignorespaces#1 #2\quad}}\bfseries#3}
\renewcommand{\tocsubsection}[3]{%
	\indentlabel{\@ifnotempty{#2}{\ignorespaces#1 #2\quad}}#3}

\newcommand\@dotsep{4.5}
\def\@tocline#1#2#3#4#5#6#7{\relax
	\ifnum #1>\c@tocdepth 
	\else
	\par \addpenalty\@secpenalty\addvspace{#2}%
	\begingroup \hyphenpenalty\@M
	\@ifempty{#4}{%
		\@tempdima\csname r@tocindent\number#1\endcsname\relax
	}{%
		\@tempdima#4\relax
	}%
	\parindent\z@ \leftskip#3\relax \advance\leftskip\@tempdima\relax
	\rightskip\@pnumwidth plus1em \parfillskip-\@pnumwidth
	#5\leavevmode\hskip-\@tempdima{#6}\nobreak
	\leaders\hbox{$\m@th\mkern \@dotsep mu\hbox{.}\mkern \@dotsep mu$}\hfill
	\nobreak
	\hbox to\@pnumwidth{\@tocpagenum{\ifnum#1=1\bfseries\fi#7}}\par
	\nobreak
	\endgroup
	\fi}
\AtBeginDocument{%
	\expandafter\renewcommand\csname r@tocindent0\endcsname{0pt}
}
\def\l@subsection{\@tocline{2}{0pt}{2.5pc}{5pc}{}}
\makeatother

\newcommand{\nn}[1]{{\left\vert\kern-0.25ex\left\vert\kern-0.25ex\left\vert #1 
		\right\vert\kern-0.25ex\right\vert\kern-0.25ex\right\vert}}

\renewcommand{\geq}{\geqslant}
\renewcommand{\leq}{\leqslant}

\newcommand{\NA}{\operatorname{NA}}

\newcommand{\e}{\varepsilon}
\newcommand{\ep}{\varepsilon}

\thanks{}

\subjclass[2020]{}

\date{\today}


\begin{document}

\title[On Holomorphic Functions Attaining their weighted norms]{On Holomorphic Functions Attaining their weighted norms}

\author[S.~Dantas]{Sheldon Dantas}
\address[S.~Dantas]{Departament de Matem\`atiques and Institut Universitari de Matem\`atiques i Aplicacions de Castell\'o (IMAC), Universitat Jaume I, Campus del Riu Sec. s/n, 12071 Castell\'o, Spain \newline
\href{https://orcid.org/0000-0001-8117-3760}{ORCID: \texttt{0000-0001-8117-3760}}}
\email{\texttt{dantas@uji.es}}

\author[R. Medina]{Rubén Medina}
\address[R. Medina]{Universidad de Granada, Facultad de Ciencias. Departamento de Análisis Matemático, 18071-Granada (Spain); and Czech Technical University in Prague, Faculty of Electrical Engineering. Department of Mathematics, Technická 2, 166 27 Praha 6 (Czech Republic) \newline
\href{https://orcid.org/0000-0002-4925-0057}{ORCID: \texttt{0000-0002-4925-0057}}}
\email{rubenmedina@ugr.es}

\thanks{}

\date{\today}

\begin{abstract} We study holomorphic functions attaining weighted norms and its connections with the classical theory of norm attaining holomorphic functions. We prove that there are polynomials on $\ell_p$ which attain their weighted but not their supremum norm and viceversa. Nevertheless, we also prove that in the context of polynomials of fixed degree both norms are in fact equivalent. This leads us to the main problem of the paper, namely, whether the holomorphic functions attaining their weighted norm are dense. Although we exhibit an example where this does not hold, as the main theorem of our paper, we prove the denseness provided the domain space is uniformly convex. In fact, we provide a Bollobás type theorem in this setting. For the proof of such a result we develop a new geometric technique.
\end{abstract}
\maketitle

\tableofcontents

\section{Introduction} 

Weighted and Bloch spaces of holomorphic functions have been a target of intense research in the recent years, \cite{Beltran, Beltran1, BGM, BLT, BLW, Jorda, Mario, M, M-R}. In fact, they are natural objects in partial differential equations, complex analysis, operator theory, and spectral theory. On the other hand, it is a classical problem in Analysis to determine whether the supremum of a bounded function is actually a maximum, \cite{BP, Bol, J, Lind}, and specially in the setting of holomorphic functions, \cite{AAGM, AK, CK, CM, Mingu}. In this paper, we investigate when the weighted norm of a holomorphic function is attained. By compactness, we are forced to work in the infinite-dimensional setting and, for that reason, we are mainly interested on weighted spaces of holomorphic functions $f:B_X \rightarrow Y$ where $X$ is an infinite complex Banach space and $Y$ is a complex arbitrary Banach space, \cite{Beltran, Beltran1, CS, GMR, Jorda, TX}. It is worth mentioning that norm-attaining problems have been previously studied in the setting of weighted Banach spaces, see \cite{BLW, MariaMartin, M-R} where the authors study norm-attaining composition operators acting on the classical Bloch and little Bloch spaces.

\subsection{Preliminaries and notation} \label{subsec:preliminares} Throughout the whole paper we only consider {\it complex} Banach spaces. We denote by $B_X$ the {\it open} unit ball of $X$, by $\overline{B_X}$ the closed unit ball of $X$, and by $S_X$ the unit sphere of $X$.

An {\it admissible weight} on a Banach space $X$ is a function $v: B_X \rightarrow \R^+$ such that $v(x)=\widetilde v(\|x\|)$ where $\widetilde v$ satisfy the following conditions:
\begin{itemize}
    \item $\widetilde v:[0,1]\to\R$ is continuous.
    \item $\widetilde v$ is strictly decreasing.
    \item $\widetilde v(1)=0$.
\end{itemize}

Given a Banach space $Y$, we will be working with $Y$-valued holomorphic functions on the open unit ball $B_X$ that belong to the weighted space $H_v(B_X; Y)$ for some admissible weight $v$. That is, $f\in H_v(B_X;Y)$ if it is holomorphic and satisfies $\sup_{x \in B_X} v(x) \|f(x)\| < \infty$. Such a space is endowed with the {\it $v$-norm} (also known as {\it weighted norm}) given by
\begin{equation*}
    \|f\|_v := \sup_{x \in B_X} v(x) \|f(x)\|\;\;\;\;\;(f\in H_v(B_X;Y)).
\end{equation*}

We are mainly interested in the subspace $\mathcal{A}_u(B_X;Y)$ of $H_v(B_X;Y)$ consisting of all holomorphic functions which are uniformly continuous on $B_X$. Differently from what happens in the classical theory of holomorphic functions which attain their supremum norm (where the supremum is attained on $S_X$), when dealing with the weighted norm of a function $f$ in $\mathcal{A}_u(B_X;Y)$ we have the following phenomenon: if $f$ attains its $v$-norm, then the $v$-norm of $f$ must be attained at a point of the open unit ball $B_X$ of the complex Banach space $X$ in consideration (see Remark \ref{remarkopen} below). This leads us to think that the study of holomorphic functions which attain their weighted norm is connected to the study of holomorphic functions attaining their supremum on smaller balls (see Lemma \ref{lemma:carac} and the next definition).

Given $s\in(0,1]$, we define the {\it $s$-norm} of a function $f$ by
\begin{equation*}
    \|f\|_s:=\sup\limits_{x\in B_X}\|f(sx)\| \;\;\;\;(f\in\mathcal{A}_u(B_X;Y)).
\end{equation*}
Let us notice that $\|\cdot\|_s$ is a complete norm in $\mathcal{A}_u(B_X;Y)$ for every $s\in(0,1]$ (see, for instance, \cite[page 634]{AAGM} and \cite[page 499]{CM}).
For the particular case $s=1$, the $1$-norm will be called simply by the {\it supremum norm} and will be denoted by the standard notation $\|\cdot\|_\infty$.

For the sake of clarity, we will be dealing only with the standard weight $v$ defined as $v(x) = 1 - \|x\|^2$ for every $x \in B_X$. However, as the reader can immediately realize, analogous results can be stated for a general admissible weight. Therefore, given $f \in \mathcal{A}_u(B_X;Y)$, we simply set 
\begin{equation*}
    \|f\|_v = \sup_{x \in B_X} (1 - \|x\|^2) \|f(x)\|.
\end{equation*}

\noindent
The following reformulation of the $v$-norm will be convenient for the upcoming computations and we will be using it without any explicit reference. For $f \in \mathcal{A}_u(B_X; Y)$, we have that 
\begin{equation*} \label{otra-forma}
\|f\|_v = \sup_{s \in [0,1]} (1 - s^2) \|f\|_s.
\end{equation*}

Before continuing, let us introduce some basic definitions and concepts we need in order to avoid the reader jumping into many different references so often. A mapping $P: X \rightarrow Y$ is an {\it $N$-homogeneous polynomial} if there exists a symmetric $N$-linear mapping $A \in \mathcal{L}(^N X; Y)$ such that $P(x) = A(x, \ldots, x)$ for every $x \in X$. We denote by $\mathcal{P}(^N X; Y)$ the Banach space of all $N$-homogeneous polynomials from $X$ into $Y$. It is convenient to know that $0$-homogeneous polynomials are the constant mappings from $X$ into $Y$. A mapping $P: X \rightarrow Y$ is a {\it polynomial of degree at most $N$} if there exist $P_k \in \mathcal{P}(^k X; Y)$, with $k=0,1,\ldots, N$, such that $P = \sum_{k=0}^N P_k$. When $P_N \not= 0$, we say that $P$ has degree $N$. The symbol $\mathcal{P}^N (X; Y)$ stands for all polynomials of degree at most $N$. Moreover, we denote by $\mathcal{P}(X, Y) = \bigcup_{N=0}^{\infty} \mathcal{P}^N(X; Y)$ the space of all polynomials from $X$ into $Y$. For a complete background on (homogeneous) polynomials, we send the referee to \cite{Mujica} and the more recent book \cite{HJ}, where we take most of all notation from.

We are now in the position of providing the precise definitions of norm-attaining holomorphic functions for the $v$- and $s$-norms.

\begin{definition} Let $f \in \mathcal{A}_u(B_X; Y)$ and $s\in(0,1]$. We say that $f$ {\it attains the} 
\begin{itemize} 
\item[(a)] {\it $s$-norm} when there exists $x_0 \in s\overline{B_X}$ such that $\|f\|_{s} = \|f(x_0)\|$. 
\item[(b)] {\it $v$-norm} when there exists $x_0 \in \overline{B_X}$ such that $\|f\|_v = (1 - \|x_0\|^2)\|f(x_0)\|$.
\end{itemize} 
\end{definition}

Despite being straightforward, for the sake of completeness, let us rapidly show the following fact.  

\begin{remark}\label{remarkopen} Let $f \in \mathcal{A}_u(B_X;Y)$ and suppose that $f$ attains its $v$-norm at some point $x_0 \in S_X$. Then we have that 
\begin{equation*}
    0 = (1 - \|x_0\|^2) \|f(x_0)\| = \|f\|_v \geq ( 1 -\|x\|^2) \|f(x)\| 
\end{equation*}
for every $x \in B_X$. Since $f$ is bounded on $B_X$, $f \equiv 0$. In other words, if $f$ attains its $v$-norm, then it must be attained at a point of $B_X$ (not in the boundary!). On the other hand, if $f$ attains its $s$-norm, then it is attained on $sS_X$ by the vector valued Maximum Modulus Principle (see, for instance, \cite[5.G, page 40]{Mujica}).
\end{remark}

Therefore, we have the following useful connection between the $v$- and $s$-norms.

\begin{lemma} \label{lemma:carac} Let $X, Y$ be complex Banach spaces. Let $f \in \mathcal{A}_u(B_X;Y)$. The following statements are equivalent.
\begin{itemize} 
\item[(1)] $f$ attains its $v$-norm. 
\item[(2)] There is $ s_0 \in (0, 1)$ such that $\|f\|_v = (1 - s_0^2) \|f\|_{s_0}$ and $f$ attains its ${s_0}$-norm.
\end{itemize} 
\end{lemma}

Given $s\in(0,1]$ and $N\in\N$, we denote by $\NA_s \mathcal{P}^{N}(X; Y)$ the set of all polynomials from $X$ into $Y$ of degree $N$ that attain the $s$-norm. In the case $s=1$, we denote $\NA_1\mathcal{P}^{N}(X;Y)$ simply by $\NA\mathcal{P}^{N}(X;Y)$. Analogously, we define the set $\NA_v \mathcal{P}^{N}(X; Y)$ of all polynomials from $X$ into $Y$ that attain their $v$-norm. One might define the corresponding $\NA$'s for the spaces $\mathcal{P}(^NX;Y)$ and $\mathcal{A}_u(B_X; Y)$ in an obvious way.

\subsection{Our results} Let us briefly summarize our results. In Section \ref{Sec2}, we show that, if we assume that the domain space $X$ is reflexive, then {\it every} bounded weakly sequentially continuous function from $X$ into $\C$ attains its $v$-norm (see Proposition \ref{Prop:reflexive-wsc}). We then proceed to prove that in the setting of $N$-homogeneous polynomials, the $v$- and $s$-norms differ up to a constant which depends on the degree $N$ of the homogeneous polynomial. In particular, we have that homogeneous polynomials attain their $v$-norm if and only if they attain their $s$-norm for some $s \in (0, 1]$ if and only if they attain their $s$-norm for every $s\in(0,1]$ (see Proposition \ref{Prop:Poly-Hom}). This tells us that the most interesting case is the non-homogeneous case and so we proceed to study it.

First, we give a sharper way to compute the $v$-norm for polynomials of a fixed degree. Using this, we are able to prove that the $v$- and $s$-norms are equivalent in $\mathcal{P}^N(X;Y)$ for every pair of Banach spaces $X$ and $Y$. We then finish Section \ref{Sec2} giving a construction of non-homogeneous polynomials which attain their $v$-norm but do not attain their supremum norm and vice-versa (see Propositions \ref{Example:non-homogeneous-poly-Pr} and
\ref{Example:non-homogeneous-Q}, respectively).

We then finally start the study of the denseness of the set $\NA_v \mathcal{P}^N(X; Y)$, and consequently of the set $\NA_v \mathcal{A}_u(B_X;Y)$, in Section \ref{Sec3}. By using a result due to Daniel Carando and Martin Mazzitelli \cite{CM}, we prove that there exists a Banach space $X$ such that the set $\NA_v\mathcal{A}_u(B_X; Y)$ is not $\|\cdot\|_{\infty}$-dense in $\mathcal{A}_u(B_X; Y)$ whenever $Y$ is strictly convex (see Theorem \ref{counterexample-denseness}). We then move to the main theorem of the paper. As one can immediately realize, in order to get a positive result about the denseness of the set $\NA_v$, it seems not to be enough emulating some of the standard techniques from norm-attaining theory (see, for instance, the proof of \cite[Theorem 1]{Lind}). For this reason, we come up with a new geometric approach (see the first two paragraphs of Section \ref{Sec3}) to prove that if $X$ is uniformly convex, then the Banach space $\mathcal{P}^N(X; Y)$ satisfies a Bollobás-type theorem for every Banach space $Y$ (see Theorem \ref{mainthdensity}) and, as an immediate consequence of it, the set $\NA_v\mathcal{A}_u(B_X; Y)$ is $\|\cdot\|_{\infty}$ dense in $\mathcal{A}_u(B_X; Y)$ (see Corollary \ref{corollary:mainthdensity}).

\section{Relations between the weighted and supremum norms} \label{Sec2}

\subsection{Some basic results}

First, we must state and proof a basic but general fact about the norm attainment of the weighted norm for reflexive spaces (and, in particular, for finite-dimensional spaces). In what follows, we say that a sequence $(x_n)_n \subseteq B_X$ is a {\it maximizing sequence} for $f: X \rightarrow Y$ for the norm $\|\cdot\|_s$ if $ \|f(x_n)\| \rightarrow \|f\|_s$.

\begin{proposition} \label{Prop:reflexive-wsc} Let $X$ be a reflexive space. Every weakly sequentially continuous function $f: X \rightarrow \C$ bounded on $B_X$ attains the $s$-norm for every $s\in(0,1]$. In particular, such an $f$ attains the $v$-norm.
\end{proposition}

\begin{proof} Let $f: X \rightarrow \C$ be a non-zero weakly sequentially continuous function and $s\in(0,1]$ arbitrary. Then, there is $(x_n)_n \subseteq sS_X$ a maximizing sequence for the norm $\|f\|_s$. By the Smulyan lemma, since $X$ is reflexive, there exists a subsequence of $(x_n)_n$, denoted by $(x_n)_n$, and $x_0 \in s\overline{B_X}$ such that $x_n \stackrel{w}{\longrightarrow} x_0$. Then, since $f$ is weakly sequentially continuous, we have that $|f(x_n)| \rightarrow |f(x_0)|$ and so $|f(x_0)|=\|f\|_s$. Finally, if $f$ attains the $s$-norm for every $s\in(0,1]$ then by Lemma \ref{lemma:carac} $f$ attains the $v$-norm.
\end{proof}

In what follows $P^N_{wsc}(X;\C)$ stands for the subspace of $P^N(X;\C)$ consisting of weakly sequentially continuous polynomials. For homogeneous polynomials $P_{wsc}(^NX;\C)$ is defined analogously.

As an immediate consequence of Proposition \ref{Prop:reflexive-wsc} and \cite[Theorem 4.1]{Mingu}, we have the following result.

\begin{corollary} \label{Corollary:wsc} Let $X$ be a reflexive space. Suppose that $\mathcal{P}^N(X; Y) = \mathcal{P}_{wsc}^N (X; Y)$. Then, 
\begin{equation*} 
\NA_v \mathcal{P}^N(X; Y) = \mathcal{P}^N (X; Y) = \NA \mathcal{P}^N (X; Y).
\end{equation*} 
\end{corollary} 

\begin{remark} Corollary \ref{Corollary:wsc} does not hold true if one of the hypothesis is removed. 
\begin{itemize} 
\item[(a)] Let us consider $X=c_0$. In this case, we have $\mathcal{P}(^N c_0; \C) = \mathcal{P}_{wsc} (^N c_0; \C)$ and hence $\mathcal{P}^N ( c_0; \C) = \mathcal{P}_{wsc}^N ( c_0; \C)$. Nevertheless, there exist homogeneous polynomials which do not attain their norm (by the James theorem). It will then follow from Proposition \ref{Prop:Poly-Hom} below that $\NA_v \mathcal{P}(^N c_0; \C) \not= \mathcal{P}(^N c_0; \C) \not= \NA \mathcal{P}(^N c_0; \C)$. \vspace{0.2cm} 
\item[(b)] Let us consider $X = \ell_p$ and $N>p$. In this case, we have $\mathcal{P}^N ( \ell_p; \C) \not= \mathcal{P}_{wsc}^N ( \ell_p; \C)$ (see, for instance, \cite{A}). In the next subsection (Propositions \ref{Example:non-homogeneous-poly-Pr} and \ref{Example:non-homogeneous-Q}) we will see that $\NA_v \mathcal{P}^N( \ell_p; \C)$,  $\mathcal{P}^N( \ell_p; \C)$ and $\NA \mathcal{P}^N( \ell_p; \C)$ are all different to each other.
\end{itemize} 
\end{remark}

We now give a complete study of the relation between the $v$- and $s$-norms in the setting of $N$-homogeneous polynomials. We have the following result.

\begin{proposition} \label{Prop:Poly-Hom} Let $X, Y$ be Banach spaces. For every $N \in \N$, there exists $\delta_N \in (0,1)$ with $\delta_N\to0$ such that, for every $s\in(0,1]$ and every $P \in \mathcal{P}(^N X; Y)$,
\begin{equation*}
\|P\|_v = \frac{\delta_N}{s^N} \|P\|_{s}.
\end{equation*}
Moveover, the following statements are equivalent. 
\begin{itemize} 
\item[(a)] $P$ attains the $v$-norm. 
\item[(b)] $P$ attains the $s$-norm for some $s\in(0,1]$. 
\item[(c)] $P$ attains the $s$-norm for every $s\in(0,1]$.
\end{itemize} 
\end{proposition}

\begin{proof} Let $N \in \N$. Let $P \in \mathcal{P}(^N X; Y)$ be fixed. For every $s\in(0,1]$, we have that 
\begin{eqnarray*}
\|P\|_v = \sup_{x \in B_X} (1 - \|x\|^2) \|P(x)\| &=& \sup_{x \in B_X} (1 - \|x\|^2) \left(\frac{\|x\|}{s}\right)^N \left\| P \left( s\frac{x}{\|x\|} \right) \right\| \\
&=& \sup_{r \in [0,1]} (r^N - r^{N+2})\frac{1}{s^N} \left\| P \right\|_s.
\end{eqnarray*}
Now, it is not difficult to see that 
\begin{equation*}
    \delta_N :=\sup_{r \in [0,1]} (r^N - r^{N+2}) = \left( \frac{N}{N+2} \right)^{\frac{N}{2}} - \left( \frac{N}{N+2} \right)^{\frac{N+2}{2}},
\end{equation*}
where the preceding supremum is attained at $r = \sqrt{\frac{N}{N+2}}$.

Suppose now that $P\in\mathcal{P}(^N X; Y)$ attains the $v$-norm. Let us prove that there is $s\in(0,1]$ such that $P$ attains the norm $\|\cdot\|_s$. Indeed, there is $x_0\in S_X$ and $r\in[0,1]$ such that $\|P\|_v=(1-r^2)\|P(rx_0)\|$. Hence, we have that
$$\|P\|_r=\sup\limits_{x\in S_X}\|P(rx)\|=\|P(rx_0)\| $$
and so $P$ attains the $s$-norm for $s=r$. On the other hand, if $P$ attains the $s$-norm for some $s\in(0,1]$, then $\|P\|_s=\|P(sx_0)\|$ for some $x_0\in S_X$. Clearly, for every $r\in(0,1]$,
$$ \|P\|_r=\left(\frac{r}{s}\right)^N\|P\|_s=\left(\frac{r}{s}\right)^N\|P(sx_0)\|=\|P(rx_0)\|. $$
This proves that $P$ attains the norm $\|\cdot\|_r$ for every $r\in(0,1]$. Finally, if $P$ attains the norm $\|\cdot\|_r$ for every $r\in (0,1]$ then there is $r_0\in(0,1)$ such that 
$$ \|P\|_v=\sup_{r\in(0,1)}(1-r^2)\|P\|_r=(1-r_0^2)\|P\|_{r_0}.$$
Thus, taking $x_0\in r_0S_X$ such that $\|P\|_{r_0}=\|P(x_0)\|$ we conclude that $P$ attains the norm $\|\cdot\|_v$ since
$\|P\|_v=(1-r_0^2)\|P\|_{r_0}=(1-\|x_0\|)\|P(x_0)\|$.
\end{proof}



Proposition \ref{Prop:Poly-Hom} above tell us in particular that everything which is done for $N$-homogeneous polynomials attaining the norm $\|\cdot\|_\infty$ applies to the norm $\|\cdot\|_v$. In other words, for every $X, Y$, and $N \in \N$, we have that 
\begin{equation*}
\NA_v \mathcal{P}(^N X; Y) = \NA \mathcal{P}(^N X; Y) 
\end{equation*} 
and also 
\begin{equation*} 
\overline{\NA \mathcal{P}(^N X; Y)}^{\|\cdot\|_{\infty}} = \mathcal{P}(^N X; Y) \;\;\Leftrightarrow\;\; \overline{\NA_v \mathcal{P}(^N X; Y)}^{\|\cdot\|_v} = \mathcal{P}(^N X; Y). 
\end{equation*} 
We send the reader to \cite{AAGM, CK} for examples where the denseness of the set $\NA \mathcal{P}(^N X; Y)$ holds true for some Banach spaces $X$ and $Y$.




\vspace{0.5cm}

\subsection{On non-homogeneous polynomials}


As a consequence of Proposition \ref{Prop:Poly-Hom} we are forced  to move forward in the search of a class of holomorphic functions where there is a difference (in terms of norm-attaining) between the sup-norm $\|\cdot\|_{\infty}$ and the norm $\|\cdot\|_v$. Having this in mind, we now change our setting to the context of non-homogeneous polynomials (see Subsection \ref{subsec:preliminares} for necessary background).

As the reader will see next, the relation between the norms $\|\cdot\|_{\infty}$ and $\|\cdot\|_v$ for $\mathcal{P}^N(X;Y)$ goes in a totally different direction. We start with the following result, which provides a convenient way on how to compute the $v$-norm of a polynomial. 

\begin{theorem} \label{otra-manera-norma-v2} For every $N\in \N$, there exists $s(N) \in (0,1)$ such that for every pair $X,Y$ of Banach spaces and every polynomial $P \in \mathcal{P}^{N}(X; Y)$, we have
\begin{equation*}
    \|P\|_v = \sup_{s \in [0, s(N)]} (1 - s^2) \|P\|_s.
\end{equation*}
\end{theorem}

In order to prove Theorem \ref{otra-manera-norma-v2}, we need the following lemma.

\begin{lemma} \label{inequalities-polynomials2} Let $X, Y$ be Banach spaces. Let $N \in \N$ be fixed. Then, for every $s \in (0, 1)$ and every $P \in \mathcal{P}^N (X; Y)$, the following holds true:
\begin{equation} \label{inequality-polynomial-1} 
\|P\|_s\geq\left(1 - \sum_{n=1}^{N} (1 - s^n) \cdot \frac{n^n}{n!} \right) \|P\|_{\infty}.
\end{equation}

\end{lemma}

\begin{proof} Let $P = \sum_{n=0}^{N} P_n \in \mathcal{P}^N (X; Y)$ be fixed where $P_n\in\mathcal{P} (^nX; Y)$ for each $n=0,1,\ldots, N$. Let $s \in (0, 1)$. For every $x \in S_X$, by \cite[Lemma 47, Chapter 1]{HJ}, we have that 

\begin{equation*}
\|P(x) - P(sx)\| = \left\| \sum_{n=1}^{N} (1-s^n) P_n(x) \right\| 
\leq \sum_{n = 1}^{N} (1 - s^n) \|P_n\|_{\infty} 
\leq \|P\|_{\infty} \sum_{n = 1}^{N} (1 - s^n) \cdot \frac{n^n}{n!}.
\end{equation*}
This means that for every $s \in (0,1)$ and $x \in S_X$, 
\begin{equation*}
\|P(sx)\| \geq \|P(x)\| - \|P(x) - P(sx)\| \\
\geq \|P(x)\| - \|P\|_{\infty} \cdot \sum_{n = 1}^{N} (1 - s^n) \cdot \frac{n^n}{n!}
\end{equation*}
which implies that 
\begin{equation*}
    \|P\|_s \geq \|P\|_{\infty} - \|P\|_{\infty} \sum_{n = 1}^{N} ( 1 - s^n) \cdot \frac{n^n}{n!} = \|P\|_{\infty} \left(1 - \sum_{n = 1}^{N} (1 - s^n) \cdot \frac{n^n}{n!} \right)
\end{equation*}
as desired.
\end{proof}

Now we can prove Theorem \ref{otra-manera-norma-v2}. 

\begin{proof}[Proof of Theorem \ref{otra-manera-norma-v2}] Let us first define the auxiliary functions $f,g:[0,1]\to\R$ given by
$$f(s)=(1-s^2)\left(1 - \sum_{n = 1}^{N} (1 - s^n) \cdot \frac{n^n}{n!} \right)\;\;\;\mbox{and}\;\;\;g(s)=1-s^2\;\;\;\;\;(s\in[0,1]).$$
Clearly, there is $z\in(0,1)$ such that
$$f(z)=\max\limits_{s\in[0,1]}f(s)>0.$$
Now, since $f(z)>0$ and $g(1)=0$, by continuity there is $s(N)\in(z,1)$ such that
\begin{equation}\label{pointineq}
    g\big(s(N)\big)\leq f(z).
\end{equation}
We will show that such an $s(N)$ satisfies the statement of the theorem. Indeed, let us consider a pair $X,Y$ of Banach spaces and a polynomial $P\in \mathcal{P}^N(X;Y)$. Then, by Lemma \ref{inequalities-polynomials2} and inequality \eqref{pointineq} we have
$$\begin{aligned}\sup\limits_{s\in[0,s(N)]}(1-s^2)\|P\|_s\ge&\|P\|_\infty\sup\limits_{s\in[0,s(N)]}f(s)=\|P\|_\infty f(z)\ge\|P\|_\infty g\big(s(N)\big)\\=&\sup\limits_{s\in[s(N),1]}(1-s^2)\|P\|_\infty\ge\sup\limits_{s\in[s(N),1]}(1-s^2)\|P\|_s.\end{aligned}$$
Therefore,
$$\|P\|_v=\sup\limits_{s\in[0,1]}(1-s^2)\|P\|_s=\sup\limits_{s\in[0,s(N)]}(1-s^2)\|P\|_s.$$
\end{proof}

Theorem \ref{otra-manera-norma-v2} will be very helpful during this section. In fact, as a consequence of it we are able to proof that the norms $\|\cdot\|_v$ and $\|\cdot\|_s$ are equivalent in $\mathcal{P}^N(X;Y)$ (see Theorem \ref{Theorem:equivalence} below). Before that, let us introduce some notation.

\vspace{0.2cm} 
\noindent 
{\bf Notation}: Let $N \in \N$ and $\alpha \in [0,1]$. Set
\begin{equation} \label{s-de-alpha-N}
    s(\alpha, N) := \left(1 - \alpha \cdot \frac{N!}{N^{N+1}} \right)^{\frac{1}{N}}.
\end{equation}
Notice that $0 < s(\alpha, N) < 1$ for every $\alpha\in(0,1)$ and $N\in \N$. Moreover, by using Stirling's approximation $N! \sim \sqrt{2\pi N} \left(\frac{N}{e}\right)^N$, we can see that $s(\alpha, N) \rightarrow 1$ as $N \rightarrow \infty$. 
\vspace{0.2cm}

Now we are in the position of proving that $\|\cdot\|_{\infty}$ and $\|\cdot\|_v$ are equivalent, even in the setting of non-homogeneous polynomials. 

\begin{theorem} \label{Theorem:equivalence} Let $X, Y$ be Banach spaces and $N \in \N$. Then, for every $\alpha \in (0, 1)$, we have 
\begin{itemize}
    \item[(1)] $\|P\|_{s(\alpha, N)} \geq (1 - \alpha) \|P\|_{\infty}$ for every $P \in \mathcal{P}^N(X;Y)$.
    \item[(2)] $\|P\|_v \geq [1 - s(\alpha, N)^2] (1 - \alpha) \|P\|_{\infty}$ for every $P \in \mathcal{P}^N(X;Y)$.
\end{itemize}
In particular, (1) and (2) imply that all the norms $\|\cdot\|_{s(\alpha,N)}$, $\|\cdot\|_{\infty}$, and $\|\cdot\|_v$ are equivalent to each other in $\mathcal{P}^N( X; Y)$. 
\end{theorem}

\begin{proof} Let $\alpha \in (0,1)$. Let $P \in \mathcal{P}^N(X; Y)$ be fixed. By using Lemma \ref{inequalities-polynomials2}, we have
\begin{equation*}
    \|P\|_{s(\alpha,N)} \geq \bigg(1-\Big[ 1 - s(\alpha, N)^N \Big] \cdot \frac{N^N}{N!} \cdot N\bigg) \cdot \|P\|_{\infty}=(1-\alpha)\|P\|_\infty
\end{equation*}
and we prove (1). To prove (2), we use (1) as follows:
\begin{eqnarray*}
\|P\|_v = \sup_{s \in [0,1]} ( 1 - s^2 ) \sup_{x \in S_X} \|P(sx)\| &\geq& \Big[1 - s(\alpha, N)^2 \Big] \sup_{x \in S_X} \|P(s(\alpha, N)x)\| \\
&=& \Big[ 1 - s(\alpha, N)^2 \Big] \|P\|_{s(\alpha, N)} \\
&\stackrel{(1)}{\geq}& \Big[ 1 - s(\alpha, N)^2 \Big] (1 - \alpha) \|P\|_{\infty}. 
\end{eqnarray*}
\end{proof}

In particular, if we are interested in tackling problems related to approximating polynomials by polynomials that attain the $v$-norm, then such an approximation can be done in any of the norms $\|\cdot\|_v$ or $\|\cdot\|_{\infty}$. On the other hand, Theorem \ref{Theorem:equivalence} is no longer true for $\mathcal{P}(X; \C)$. Indeed, we have the following simple example.

\begin{remark} \label{remark:theorem-equiv} Let $N \in \N$ and fix $x^* \in S_{X^*}$. Define $f_N: X \rightarrow \C$ by 
\begin{equation*}
    f_N(x) := x^*(x)^N \ \ \ (x \in X).
\end{equation*}
Then, $f_N$ is a holomorphic function (in fact, it is a homogeneous polynomial) for each $N \in \N$. Notice that $\|f_N\|_{\infty} = 1$ and 
\begin{equation*}
\|f_N\|_v = \sup_{s \in [0,1]}(1 - s^2) \sup_{x \in S_X} |x^*(sx)^N| = \sup_{s \in [0,1]} (1 - s^2) s^N.   
\end{equation*}
This last supremum we have calculated already (see the proof of Proposition \ref{Prop:Poly-Hom}) and therefore
\begin{equation*}
\|f_N\|_v = \sup_{s \in [0,1]} (1 - s^2) s^N = \left( \frac{N}{N+2}\right)^{\frac{N}{2}} - \left(\frac{N}{N+2}\right)^{\frac{N+2}{2}} \longrightarrow 0 \ \mbox{as} \ N \rightarrow \infty. 
\end{equation*}
This shows that $\|\cdot\|_{\infty}$ and $\|\cdot\|_v$ cannot be equivalent on $\mathcal{P}(X; \C) := \bigcup_{N \in \N} \mathcal{P}^N (X; \C)$. 
\end{remark}


\vspace{0.5cm}

Our next aim is to construct a non-homogeneous polynomial which attains the $v$-norm but not the sup-norm $\|\cdot\|_{\infty}$ (see Proposition \ref{Example:non-homogeneous-poly-Pr} below). To do so, we need first the following auxiliary result. 

\begin{lemma} \label{Lemma:Pr} Let $1 \leq p < \infty$. For every $r \in (0, 1)$ and for every $k \in \N$ with $k \geq p$, there exists a polynomial $P_r \in \mathcal{P}^{k+1}(\ell_p; \C)$ such that 
\begin{itemize}
    \item[(1)] for every $0\leq s \leq r$, $P_r$ attains its supremum over $sB_{\ell_p}$. 
    \item[(2)] for every $s > r$, $P_r$ does not attain its supremum over $sB_{\ell_p}$. 
\end{itemize}
\end{lemma}

\begin{proof} Let $\{e_n: n \in \N \}$ be the canonical basis of $\ell_p$. Take $r \in (0, 1)$ and $k \in \N$ with $k \geq p$. Define $P_r \in \mathcal{P}^{k+1}(\ell_p; \C)$ by 
\begin{equation*} 
P_r(x) := \sum_{n \in \N} \left( 1 + \frac{r^{-k}}{n} \right) x_n^k + \left(1 - \frac{r^{-k-1}}{n} \right) x_n^{k+1} \ \ \ (x \in \ell_p).
\end{equation*}
Let us prove (1). Let $s \leq r$ and take $x \in S_{\ell_p}$ arbitrary. Since $\|x\|_k^k \leq \|x\|_p^k = 1$ and $s/r \leq 1$, we have that 
\begin{eqnarray*}
|P_r(sx)| &\leq& \sum_{n \in \N} \left( s^k + \frac{(s/r)^k}{n} \right) |x_n|^k + \left( s^{k+1} - \frac{(s/r)^{k+1}}{n} \right) |x_n|^{k+1} \\
&\leq& \sum_{n \in \N} \left[ s^k + s^{k+1} + \frac{1}{n} \left( \left(\frac{s}{r}\right)^k - \left( \frac{s}{r} \right)^{k+1} \right) \right] |x_n|^{k} \\
&\leq& \left[ s^k + s^{k+1} +  \left(\frac{s}{r}\right)^k - \left( \frac{s}{r} \right)^{k+1}  \right] \|x\|_k^k \\
&\leq& s^k + s^{k+1} + \left( \frac{s}{r} \right)^k - \left( \frac{s}{r} \right)^{k+1}. 
\end{eqnarray*}
This shows that 
\begin{equation*}
 \sup_{x \in S_{\ell_p}} |P_r(sx)| \leq s^k + s^{k+1} + \left( \frac{s}{r} \right)^k - \left( \frac{s}{r} \right)^{k+1}.
\end{equation*}
On the other hand, 
\begin{equation*}
P_r(s e_1) = s^k + s^{k+1} + \left( \frac{s}{r} \right)^k - \left( \frac{s}{r} \right)^{k+1}.
\end{equation*}
This proves (1). Finally, let us move on to (2). Assume that $s > r$. On the one hand, notice that 
\begin{equation*}
P_r(s e_n) = s^k + s^{k+1} + \frac{1}{n} \left( \left( \frac{s}{r} \right)^k - \left( \frac{s}{r} \right)^{k+1} \right) \rightarrow s^k + s^{k+1} \ \ \mbox{as} \ \ n \rightarrow \infty. 
\end{equation*}
This shows that $\sup_{x \in S_{\ell_p}} P_r(sx) \geq s^k + s^{k+1}$. On the other hand, we will prove that this is in fact the supremum and that it is never attained. Indeed, let $x \in S_{\ell_p}$ be arbitrary. Since $\|x\|_k^k \leq \|x\|_p^k = 1$ and $s/r > 1$, we have that 
\begin{equation*}
    |P_r(sx)| \leq \sum_{n \in \N} \left[ s^k + s^{k+1} + \frac{1}{n} \left( \left( \frac{s}{r} \right)^k - \left( \frac{s}{r} \right)^{k+1} \right) \right] |x_n|^k < (s^k + s^{k+1}) \|x\|_k^k \leq s^k + s^{k+1}.
\end{equation*}
This proves (2) and we are done.
\end{proof}

Now we are ready to provide the desired example. In order to do this, we invoke Theorem \ref{otra-manera-norma-v2} and Lemma \ref{Lemma:Pr}. 

\begin{proposition} \label{Example:non-homogeneous-poly-Pr} Let $1 \leq p < \infty$ and $k \in \N$ with $k \geq p$. Then, there exists a non-homogeneous polynomial $P \in \mathcal{P}^{k+1} (\ell_p; \C)$ such that $P \in \NA_v \mathcal{P}^{k+1}(\ell_p; \C)$ but $P \not\in \NA \mathcal{P}^{k+1}(\ell_p;\C)$.  
\end{proposition}

\begin{proof} Let $1 \leq p < \infty$ and take $k \in \N$ with $k \geq p$. Let $s(k+1)$ be the quantity given by Theorem \ref{otra-manera-norma-v2}. For $s(k+1) \leq r < 1$, we consider the positive non-homogeneous polynomial $P_r \in \mathcal{P}^{k+1}(\ell_p; \C)$ as in Lemma \ref{Lemma:Pr}. By Theorem \ref{otra-manera-norma-v2}, we have that 
\begin{equation*}
\|P_r\|_v = \sup_{s \in \left[ 0, s(k+1) \right]} (1 - s^2) \|P_r\|_s.  
\end{equation*}
By compactness, there exists $s_0 \in \left[ 0, s(k+1) \right]$ such that 
\begin{equation*}
\|P_r\|_v =  (1 - s_0^2) \|P_r\|_{s_0}. 
\end{equation*}
Now, since $s_0 \leq s(k+1) \leq r$, by item (1) of Lemma \ref{Lemma:Pr}, there exists $x_0 \in s_0 S_{\ell_p}$ such that 
\begin{equation*}
    \|P_r\|_{s_0} = |P_r(x_0)| 
\end{equation*}
Therefore, 
\begin{eqnarray*}
    \|P_r\|_v =  (1 - s_0^2) \|P_r\|_{s_0} &=& (1 - \|x_0\|^2) |P_r(x_0)|.
\end{eqnarray*}
This shows that $P_r \in \NA_v \mathcal{P}^{k+1}(\ell_p; \C)$. On the other hand, taking $s=1$ in item (2) of Lemma \ref{Lemma:Pr}, we see that $P_r \not\in \NA \mathcal{P}^{k+1}(\ell_p; \C)$.
\end{proof}

Using a similar but simpler argument we can construct a non-homogeneous polynomial $Q \in \mathcal{P}^{k+1}(\ell_p; \C)$ which belongs to $\NA \mathcal{P}^{k+1}(\ell_p; \C)$ but not to $\NA_v \mathcal{P}^{k+1}(\ell_p; \C)$ whenever $k \geq p$. Indeed, we have the following result.

\begin{proposition} \label{Example:non-homogeneous-Q} Let $1 \leq p < \infty$. For every $k \in \N$ with $k \geq p$, there exists $Q \in \mathcal{P}^{k+1}(\ell_p; \C)$ such that 
\begin{itemize}
\item[(1)] $Q \in \NA \mathcal{P}^{k+1} (\ell_p; \C)$
    \item[(2)] For every $0< s < 1$, $Q$ does not attain the $s$-norm. In particular, $Q \not\in \NA_v \mathcal{P}^{k+1}(\ell_p; \C)$.
\end{itemize}
\end{proposition}

\begin{proof} Let $\{e_n: n \in \N\}$ be the canonical basis of $\ell_p$. Take $k \in \N$ to be such that $k \geq p$. Define $Q \in \mathcal{P}^{k+1}(\ell_p; \C)$ by 
\begin{equation*}
 Q(x) := \sum_{n \in \N} \left(1 - \frac{1}{n} \right) x_n^k + \left(1 + \frac{1}{n} \right) x_n^{k+1} \ \ \ (x \in \ell_p). 
\end{equation*}
Clearly $Q(0)=0$. For every $x \in S_{\ell_p}$, since $\|x\|_k^k \leq \|x\|_p^k = 1$, we have that 
\begin{equation*}
|Q(x)| \leq \sum_{n \in \N} \left(1 - \frac{1}{n}\right) |x_n|^k + \left(1 + \frac{1}{n} \right) |x_n|^{k+1} \leq 2 \|x\|_k^k \leq 2
\end{equation*}
and since $Q(e_n) = 2$ for every $n \in \N$, we have that (1) holds true. Now, let $s \in (0, 1)$. Since $s^{k+1} - s^k < 0$, we have that 
\begin{equation*}
|Q(e_n)| = \left(1 - \frac{1}{n}\right) s^k + \left(1 + \frac{1}{n}\right) s^{k+1} = s^k + s^{k+1} + \frac{s^{k+1}-s^k}{n} \rightarrow s^{k} + s^{k+1}
\end{equation*}
as $n \rightarrow \infty$. This shows that $\sup_{x \in S_{\ell_p}} |Q(sx)| \geq s^k + s^{k+1}$. Now, let $x \in S_{\ell_p}$ be arbitrary. Since $\|x\|_k^k \leq \|x\|_p^k = 1$ and $s^{k+1} - s^k < 0$,     we have 
\begin{eqnarray*}
|Q(sx)| &\leq& \sum_{n \in \N} \left[ \left(1 - \frac{1}{n} \right) s^k + \left(1 + \frac{1}{n}\right) s^{k+1} \right] |x_n|^k \\
&\leq& s^k + s^{k+1} + \left(\frac{s^{k+1} - s^k}{n}\right) \\
&<& s^k + s^{k+1} 
\end{eqnarray*}
which proves that $\sup_{x \in S_{\ell_p}} |Q(sx)| \leq s^k + s^{k+1}$ and it is never attained. This proves (2). Therefore, for every $x \in B_{\ell_p}$ with $x \not= 0$ we have that 
\begin{equation*}
(1 - \|x\|^2) |Q(x)|\stackrel{(2)}{<}(1 - \|x\|^2) \|Q\|_{\|x\|} \leq  \|Q\|_v 
\end{equation*}
and now we are done. 
\end{proof}

\section{Denseness of the set $\NA_v \mathcal{A}_u(B_X;Y)$} \label{Sec3}

Our goal in this section is to find classes of Banach spaces $X$ such that the set $\NA_v \mathcal{P}^{N} (X; \C)$ is dense in $\mathcal{P}^{N} (X; \C)$. As a consequence, we will be studying also when the set $\NA_v \mathcal{A}_u (B_X; \C)$ is $\|\cdot\|_{\infty}$ norm dense in $\mathcal{A}_u (B_X; \C)$ .

We will show that this density does {\it not} hold in general (see Theorem \ref{counterexample-denseness}) but it does when $X$ is taken to be uniformly convex. In fact, for every Banach space $Y$ and every $N\in\N$, we show that 
\begin{equation*}
\overline{\NA_v \mathcal{P}^{N} (X; Y)}=\mathcal{P}^{N} (X; Y)
\end{equation*} 
whenever $X$ is uniformly convex (see Theorem \ref{mainthdensity}), where the closure is taken either with the $v$-norm or supremum norm thanks to Theorem \ref{Theorem:equivalence}. The issue here is that weighted norms can be attained at any point of the open unit ball (including the origin!) while with the supremum norm it is just a matter of finding a point in the unit sphere. There are several ways to face this difficulty. The first one that we try is to approximate a polynomial $P$ by a polynomial $Q$ attaining its supremum over {\it all} unit spheres of $X$ (see Corollary \ref{corolfinite} for such a solution in the setting of $C(K)$ with $K$ scattered). Whether this can be obtained in a more general setting is still unknown to us. The second possible approach is to find $s\in(0,1)$ such that
\begin{equation*} 
\|P\|_v=(1-s^2)\|P\|_s
\end{equation*} 
and try to approximate $P$ by some polynomial $Q$ attaining its $s$-norm. In fact, by \cite[Theorem 3.1]{AAGM}, this is possible when $X$ satisfies the Radon Nikodým property. The outcome from this approach is that it may occur that
\begin{equation*} 
\|Q\|_v>(1-s^2)\|Q\|_s.
\end{equation*} 
Therefore, we decide to use a completely new approach, using the differences between the weighted and supremum norm in our favor.

\begin{figure}
    \centering

\tikzset{every picture/.style={line width=0.75pt}} 

\begin{tikzpicture}[x=0.75pt,y=0.75pt,yscale=-1,xscale=1]

\draw   (191,201) .. controls (191,123.68) and (253.68,61) .. (331,61) .. controls (408.32,61) and (471,123.68) .. (471,201) .. controls (471,278.32) and (408.32,341) .. (331,341) .. controls (253.68,341) and (191,278.32) .. (191,201) -- cycle ;
\draw  [fill={rgb, 255:red, 0; green, 0; blue, 0 }  ,fill opacity=1 ] (327.9,200.5) .. controls (327.9,199.06) and (329.06,197.9) .. (330.5,197.9) .. controls (331.94,197.9) and (333.1,199.06) .. (333.1,200.5) .. controls (333.1,201.94) and (331.94,203.1) .. (330.5,203.1) .. controls (329.06,203.1) and (327.9,201.94) .. (327.9,200.5) -- cycle ;
\draw  [color={rgb, 255:red, 0; green, 0; blue, 0 }  ,draw opacity=1 ][dash pattern={on 4.5pt off 4.5pt}][line width=0.75]  (216.44,281.14) .. controls (183.59,233.71) and (208.45,159.6) .. (271.96,115.6) .. controls (335.48,71.6) and (413.6,74.38) .. (446.45,121.81) .. controls (479.31,169.24) and (454.45,243.35) .. (390.93,287.35) .. controls (327.42,331.35) and (249.3,328.57) .. (216.44,281.14) -- cycle ;
\draw    (467.56,105.34) -- (190.84,297.06) ;
\draw [shift={(189.2,298.2)}, rotate = 325.28] [color={rgb, 255:red, 0; green, 0; blue, 0 }  ][line width=0.75]    (10.93,-3.29) .. controls (6.95,-1.4) and (3.31,-0.3) .. (0,0) .. controls (3.31,0.3) and (6.95,1.4) .. (10.93,3.29)   ;
\draw [shift={(469.2,104.2)}, rotate = 145.28] [color={rgb, 255:red, 0; green, 0; blue, 0 }  ][line width=0.75]    (10.93,-3.29) .. controls (6.95,-1.4) and (3.31,-0.3) .. (0,0) .. controls (3.31,0.3) and (6.95,1.4) .. (10.93,3.29)   ;
\draw  [color={rgb, 255:red, 0; green, 0; blue, 0 }  ,draw opacity=1 ][fill={rgb, 255:red, 0; green, 0; blue, 0 }  ,fill opacity=1 ] (421.06,136.02) .. controls (421.06,134.75) and (422.08,133.72) .. (423.35,133.72) .. controls (424.62,133.72) and (425.64,134.75) .. (425.64,136.02) .. controls (425.64,137.28) and (424.62,138.31) .. (423.35,138.31) .. controls (422.08,138.31) and (421.06,137.28) .. (421.06,136.02) -- cycle ;
\draw    (331.33,65.47) -- (341.5,79.64) ;
\draw [shift={(342.67,81.27)}, rotate = 234.35] [color={rgb, 255:red, 0; green, 0; blue, 0 }  ][line width=0.75]    (7.65,-2.3) .. controls (4.86,-0.97) and (2.31,-0.21) .. (0,0) .. controls (2.31,0.21) and (4.86,0.98) .. (7.65,2.3)   ;
\draw    (455.23,221.71) -- (463.2,232) ;
\draw [shift={(454,220.13)}, rotate = 52.21] [color={rgb, 255:red, 0; green, 0; blue, 0 }  ][line width=0.75]    (7.65,-2.3) .. controls (4.86,-0.97) and (2.31,-0.21) .. (0,0) .. controls (2.31,0.21) and (4.86,0.98) .. (7.65,2.3)   ;
\draw    (425.83,267.42) -- (437.67,283.87) ;
\draw [shift={(424.67,265.8)}, rotate = 54.26] [color={rgb, 255:red, 0; green, 0; blue, 0 }  ][line width=0.75]    (7.65,-2.3) .. controls (4.86,-0.97) and (2.31,-0.21) .. (0,0) .. controls (2.31,0.21) and (4.86,0.98) .. (7.65,2.3)   ;
\draw    (370.17,310.09) -- (382,326.53) ;
\draw [shift={(369,308.47)}, rotate = 54.26] [color={rgb, 255:red, 0; green, 0; blue, 0 }  ][line width=0.75]    (7.65,-2.3) .. controls (4.86,-0.97) and (2.31,-0.21) .. (0,0) .. controls (2.31,0.21) and (4.86,0.98) .. (7.65,2.3)   ;
\draw    (303.58,325.33) -- (310.67,335.93) ;
\draw [shift={(302.47,323.67)}, rotate = 56.24] [color={rgb, 255:red, 0; green, 0; blue, 0 }  ][line width=0.75]    (7.65,-2.3) .. controls (4.86,-0.97) and (2.31,-0.21) .. (0,0) .. controls (2.31,0.21) and (4.86,0.98) .. (7.65,2.3)   ;
\draw  [color={rgb, 255:red, 0; green, 0; blue, 0 }  ,draw opacity=1 ][fill={rgb, 255:red, 0; green, 0; blue, 0 }  ,fill opacity=1 ] (243.59,90.02) .. controls (243.59,88.75) and (244.62,87.72) .. (245.88,87.72) .. controls (247.15,87.72) and (248.17,88.75) .. (248.17,90.02) .. controls (248.17,91.28) and (247.15,92.31) .. (245.88,92.31) .. controls (244.62,92.31) and (243.59,91.28) .. (243.59,90.02) -- cycle ;
\draw  [color={rgb, 255:red, 0; green, 0; blue, 0 }  ,draw opacity=1 ][fill={rgb, 255:red, 0; green, 0; blue, 0 }  ,fill opacity=1 ] (264.26,119.02) .. controls (264.26,117.75) and (265.28,116.72) .. (266.55,116.72) .. controls (267.82,116.72) and (268.84,117.75) .. (268.84,119.02) .. controls (268.84,120.28) and (267.82,121.31) .. (266.55,121.31) .. controls (265.28,121.31) and (264.26,120.28) .. (264.26,119.02) -- cycle ;
\draw    (249.67,95.8) -- (261.5,112.24) ;
\draw [shift={(262.67,113.87)}, rotate = 234.26] [color={rgb, 255:red, 0; green, 0; blue, 0 }  ][line width=0.75]    (7.65,-2.3) .. controls (4.86,-0.97) and (2.31,-0.21) .. (0,0) .. controls (2.31,0.21) and (4.86,0.98) .. (7.65,2.3)   ;
\draw    (205.67,149.47) -- (215.78,163.7) ;
\draw [shift={(216.93,165.33)}, rotate = 234.62] [color={rgb, 255:red, 0; green, 0; blue, 0 }  ][line width=0.75]    (7.65,-2.3) .. controls (4.86,-0.97) and (2.31,-0.21) .. (0,0) .. controls (2.31,0.21) and (4.86,0.98) .. (7.65,2.3)   ;

\draw (459,274.6) node [anchor=north west][inner sep=0.75pt]    {$S_{X}$};
\draw (339.1,194.5) node [anchor=north west][inner sep=0.75pt]    {$0$};
\draw (415.73,114.87) node [anchor=north west][inner sep=0.75pt]    {$x$};
\draw (255.73,164.6) node [anchor=north west][inner sep=0.75pt]    {$S( y) =R( z)$};
\draw (227.47,74) node [anchor=north west][inner sep=0.75pt]    {$y$};
\draw (273.96,119) node [anchor=north west][inner sep=0.75pt]    {$z$};

\end{tikzpicture}
    \caption{The linear transformation of $S_X$ onto an ellipsoid aligned with an element $x\in B_X$.}
    \label{elipses0}
\end{figure}
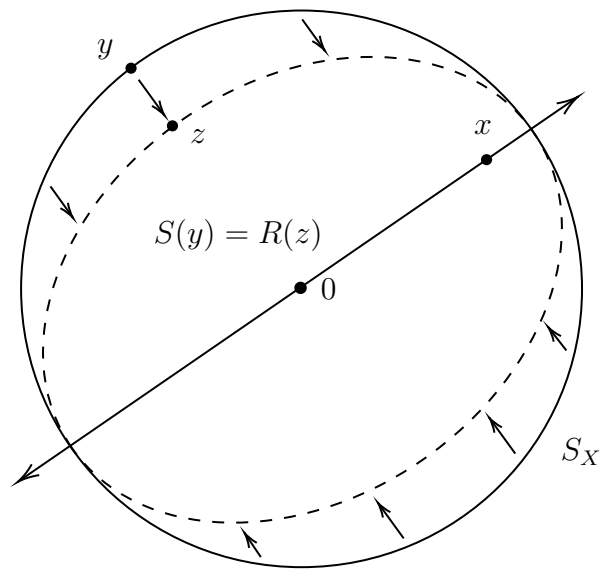

The way we approximate $P$ is by the limit of a sequence of polynomials such that we control the geometric place where these polynomials almost attain the $v$-norm. The mentioned sequence is constructed inductively using the following method (see Lemma \ref{estimate}): given a polynomial $R$, we first find a point $x$ where $R$ almost attain its weighted norm and then approximate $R$ by the polynomial $S$ whose values over the spheres of $X$ are the values of $R$ over ellipsoids aligned with the element $x$ (see Figure \ref{elipses0}). This idea will become clear in the development of the whole proof as the reader can see in Subsection \ref{subsec:main}. 

\subsection{A counterexample} We start by showing that there are Banach spaces $X$ and $Y$ such that the set $\NA_v \mathcal{A}_u(B_X;Y)$ is not $\|\cdot\|_{\infty}$ norm dense in $\mathcal{A}_u (X; Y)$. In fact, the space taken to be the domain space is the predual of a Lorentz sequence space. 

An admissible sequence will mean a decreasing sequence $(w_n)_{n \in \N} \subset \R^+$ such that $w_1=1$ and $\lim w_n=0$. The Lorentz sequence space $d(w,1)$ associated to an admissible sequence $w=(w_n)_{n \in \N}$ will be the space of all bounded sequences $x=(x_n)_{n \in \N} \subseteq \C$ such that
$$\|x\|_{w,1}=\sum \limits_{n\in\N} x^*_nw_n<\infty,$$
where $(x^*_n)_{n \in \N} \subseteq \R$ is the decreasing rearrangement of $(|x_n|)_{n \in \N}$. The space $d(w,1)$ is a reflexive Banach space when endowed with the norm $\|\cdot\|_{w,1}$. We will denote the predual of $d(w,1)$ as $d_*(w,1)$.

We will need the following known result due to D. Carando and M. Mazzitelli. We will rephrase it here for convenience.

\begin{theorem} \cite[Proposition 4.10]{CM} \label{theorem:CM} Let $N \geq 2$ and $w \in \ell_N$ be an admissible sequence. Let $X =  d_{*}(w,1)$ and $Y$ be a strictly convex Banach space. There exists an $N$-homogeneous polynomial $Q \in \mathcal{P}(^NX; Y)$ such that whenever $f \in \mathcal{A}_u(B_X; Y)$ attains its $s$-norm $\|\cdot\|_s$ for some $0 < s \leq 1$, we have \begin{equation*} 
\|Q - f\|_{\infty} \geq (1 - s)^N.
\end{equation*} 
\end{theorem}

At this point, it is worth mentioning it seems to be unknown whether a Bishop-Phelps theorem (see \cite{BP} or, for instance, \cite[Theorem 7.41]{FHHMZ}) holds for $\mathcal{A}_u(B_X; \C)$ with the sup-norm. In other words, we do not know whether the set $\NA\mathcal{A}_u(B_X;\C)$ is $\|\cdot\|_{\infty}$ dense in $\mathcal{A}_u(B_X; \C)$ for every Banach space $X$. Nevertheless, the following result follows as a combination of Theorem \ref{theorem:CM} and Theorem \ref{Theorem:equivalence}.(1). 

\begin{corollary} Let $N \geq 2$ and $w \in \ell_N$ be an admissible sequence. Then,
\begin{equation*}
    \overline{\NA\mathcal{P}^N (d_{*}(w,1); \C)}^{\|\cdot\|_{\infty}} \not= \mathcal{P}^N(d_{*}(w,1); \C).
\end{equation*}
\end{corollary}

Now we are ready to present our counterexample.

\begin{theorem} \label{counterexample-denseness} Let $N \geq 2$ and $w \in \ell_N$ be an admissible sequence. Let $X = d_{*}(w, 1)$ and $Y$ be a strictly convex Banach space. There exists an $N$-homogeneous polynomial $Q \in \mathcal{P}(^N X;Y)$ such that 
\begin{equation*} 
Q \not\in \overline{\NA_v\mathcal{A}_u (B_X; Y)}^{\|\cdot\|_{\infty}}.
\end{equation*} 
In particular, the set $\NA_v\mathcal{A}_u(B_X; Y)$ is not $\|\cdot\|_{\infty}$-dense in $\mathcal{A}_u(B_X; Y)$.
\end{theorem}

\begin{proof} Indeed, let $Q$ be the one given in \cite[Proposition 4.10]{CM}. By contradiction, let us suppose that $Q \in \overline{\NA_v\mathcal{A}_u (X; Y)}^{\|\cdot\|_{\infty}}$. Then, there exists $(f_n)_{n \in \N} \subseteq \NA_v\mathcal{A}_u(X; Y)$ such that $\|f_n - Q\|_{\infty} \rightarrow 0$ as $n \rightarrow \infty$. We have that 
\begin{equation*}
    \|Q\|_v = \sup_{s \in [0,1]} (1 - s^2) s^N \cdot \|Q\|_{\infty} = \|Q\|_{\infty} \cdot \sup_{s \in [0,1]} (1-s^2)s^N.
\end{equation*}

\noindent
{\bf Claim}: For every $\e > 0$, there exists $\rho(\e) > 0$ such that 
\begin{equation} \label{counter1}
    \|Q\|_v - \rho(\e) > \sup_{s \in [0,1] \setminus \left[ \sqrt{\frac{N}{N+2}} - \e, \sqrt{\frac{N}{N+2}} + \e \right]} (1 - s^2) \|Q\|_s. 
\end{equation}
\noindent
\begin{proof}[Proof of the Claim] Recall from the proof of Proposition \ref{Prop:Poly-Hom} that $\sup_{s \in [0,1]} (1-s^2)s^N$ is attained only at $\sqrt{\frac{N}{N+2}}$. Having this in mind, it is clear that, for every $\e>0$, there exists $\tilde{\rho}(\e) > 0$ such that 
\begin{equation} \label{counter2}
    \sup_{s \in [0,1]} (1 - s^2) s^N - \tilde{\rho}(\e) > \sup_{s \in [0,1] \setminus \left[ \sqrt{\frac{N}{N+2}} - \e, \sqrt{\frac{N}{N+2}} + \e \right]} (1 - s^2) s^N.
\end{equation}
Set $\rho(\e):= \tilde{\rho}(\e) \|Q\|_{\infty} > 0$. Then, taking into account that $Q$ is a $N$-homogeneous polynomial,
\begin{eqnarray*}
\|Q\|_v - \rho(\e) &=& \sup_{s \in [0,1]} (1 - s^2) s^N \|Q\|_{\infty} - \tilde{\rho(\e)}\|Q\|_{\infty} \\
&=& \|Q\|_{\infty} \cdot \left[ \sup_{s \in [0,1]} (1 - s^2) s^N - \tilde{\rho}(\e) \right] \\
&\stackrel{(\ref{counter2})}{>}&  \|Q\|_{\infty} \cdot \sup_{s \in [0,1] \setminus \left[ \sqrt{\frac{N}{N+2}} - \e, \sqrt{\frac{N}{N+2}} + \e \right]} (1 - s^2) s^N \\
&=& \sup_{s \in [0,1] \setminus \left[ \sqrt{\frac{N}{N+2}} - \e, \sqrt{\frac{N}{N+2}} + \e \right]} (1 - s^2) \|Q\|_s. 
\end{eqnarray*} 
\end{proof} 

Let us take $\e > 0$ to be such that 
\begin{equation} \label{counter4}
\sqrt{\frac{N}{N+2}} + \e < 1
\end{equation}
There exists $n(\e) \in \N$ such that 
\begin{equation*}
    \|f_n - Q\|_{\infty} < \frac{\rho(\e)}{2}
\end{equation*}
for every $n \geq n(\e)$. We observe that $\|f_n - Q\|_v < \frac{\rho(\e)}{2}$ and $\|f_n - Q\|_s <\frac{\rho(\e)}{2}$ for every $0 < s \leq 1$ for every $n \geq n(\e)$. So,
\begin{equation} \label{counter3}
\|f_n\|_v > \|Q\|_v - \frac{\rho(\e)}{2} \ \ \ \mbox{and} \ \ \ \|f_n\|_s < \|Q\|_s + \frac{\rho(\e)}{2}
\end{equation}
for every $n \geq n(\e)$ and $0 < s \leq 1$. Hence, if $n\ge n(\ep)$ then,
\begin{eqnarray*}
\|f_n\|_v \stackrel{(\ref{counter3})}{>} \|Q\|_v - \frac{\rho(\e)}{2} &=& \left(\|Q\|_v - \rho(\e)\right) + \frac{\rho(\e)}{2} \\
&\stackrel{(\ref{counter1})}{>}& \sup_{s \in [0,1] \setminus \left[ \sqrt{\frac{N}{N+2}} - \e, \sqrt{\frac{N}{N+2}} +\e \right]} (1 - s^2) \|Q\|_s+ \frac{\rho(\e)}{2} \\
&>& \sup_{s \in [0,1] \setminus \left[ \sqrt{\frac{N}{N+2}} - \e, \sqrt{\frac{N}{N+2}} +\e \right]} (1 - s^2) \left( \|Q\|_s + \frac{\rho(\e)}{2} \right) \\
&\stackrel{(\ref{counter3})}{>}& \sup_{s \in [0,1] \setminus \left[ \sqrt{\frac{N}{N+2}} - \e, \sqrt{\frac{N}{N+2}} +\e \right]} (1 - s^2) \|f_n\|_s.
\end{eqnarray*}
Since $f_n \in \NA_vA_u(B_X; Y)$, by Lemma \ref{lemma:carac} there exists $s_n \in (0, 1]$ such that $\|f_n\|_v = (1 - s_n^2) \|f_n\|_{s_n}$ and $f_n$ attains its norm $\|\cdot\|_{s_n}$. This implies, having in mind the previous estimation for $\|f_n\|_v$, that 
\begin{equation*}
    s_n \in \left[ \sqrt{\frac{N}{N+2}} - \e, \sqrt{\frac{N}{N+2}} + \e \right] 
\end{equation*}
for every $n \geq n(\e)$. In particular, for every $n \geq n(\e)$, we have that $s_n \leq \sqrt{\frac{N}{N+2}} + \e$. Since $f_n$ attains its $s_n$-norm for $n \geq n(\e)$, by Theorem \ref{theorem:CM}, we have that 
\begin{equation*}
\|f_n - Q\|_{\infty} \geq (1 - s_n)^N > \left(1 - \sqrt{\frac{N}{N+2}} - \e \right)^N \stackrel{(\ref{counter4})}{>} 0,
\end{equation*}
which yields a contradiction.
\end{proof}

\subsection{The main theorem} \label{subsec:main} We provide now the main result of the paper. We show that a Bollobás-type theorem (see \cite{Bol} for the classical Bollobás theorem) holds for the weighted norm in the context of uniformly convex Banach spaces. More precisely, we have the following result. 

\begin{theorem}\label{mainthdensity} Let $X$ be a uniformly convex Banach space, $Y$ an arbitrary Banach space, and $N\in\N$. Then, for every $\ep>0$, there is $\eta(\ep)>0$ such that whenever $P\in\mathcal{P}^N(X;Y)$ with $\|P\|_v=1$ and $x\in B_X$ satisfy
\begin{equation*} 
(1-\|x\|^2)\|P(x)\|\ge\|P\|_v-\eta(\ep),
\end{equation*} 
there are $Q\in \mathcal{P}^N(X;Y)$ and $y\in B_X$ such that 
\begin{equation*} 
\|Q\|_v = (1-\|y\|^2)\|Q(y)\|, \ \ \ \ \ d\big(y,\text{span}_\C(x)\big)\leq\ep, \ \ \ \ \ \mbox{and} \ \ \ \ \  \|P-Q\|_\infty\leq\ep.
\end{equation*} 
\end{theorem}

As an immediate consequence of Theorem \ref{mainthdensity}, we have the following. 

\begin{corollary} \label{corollary:mainthdensity}
Let $X$ be a uniformly convex Banach space and $Y$ an arbitrary Banach space. Then, every element in $\mathcal{A}_u(B_X; Y)$ can be approximated with the $\|\cdot\|_{\infty}$ norm by elements in $\mathcal{A}_u(B_X;Y)$ which attain their $v$-norm. In other words, the following equality holds true
\begin{equation*} 
\overline{\NA_v\mathcal{A}_u(B_X;Y)}^{\|\cdot\|_\infty}=\mathcal{A}_u(B_X;Y).
\end{equation*} 
\end{corollary}

\vspace{0.4cm} 

Our aim from now on is to give the proof of Theorem \ref{mainthdensity}. In order to do so, we split the proof into several results, which might have their own interest.  

Taking into account (\ref{s-de-alpha-N}), let us set, for every $N\in\N$, the quantity
\begin{equation} \label{M-N}
M_N :=8(1-s(1/2,N)^2)^{-1}\sum\limits_{n=1}^N\frac{2^{2n-1}n^n}{n!}.
\end{equation} 
We start with a simple consequence of Theorem \ref{Theorem:equivalence}.

\begin{proposition}\label{lippol}
Let $X,Y$ be Banach spaces and $N\in\N$. For every $P\in \mathcal{P}^N(X;Y)$ we have
$$\|P\|_{\text{Lip}(B_X)}\leq\|P\|_vM_N/4.$$
\end{proposition}
\begin{proof} Let $P=\sum\limits_{n=0}^NP_n$ where $P_n\in\mathcal{P}(^nX;Y)$.
By Proposition 4 and Lemma 47 of \cite[Chapter 1]{HJ}, we have that
$$\|P\|_{\text{Lip}(B_X)}\leq\sum\limits_{n=1}^N\|P_n\|_{\text{Lip}(B_X)}\leq\sum\limits_{n=1}^N2^{2n-1}\|P_n\|_\infty\leq\|P\|_\infty\sum\limits_{n=1}^N\frac{2^{2n-1}n^n}{n!}.$$
Now, from  Theorem \ref{Theorem:equivalence} with $\alpha=1/2$ we conclude that
$$\|P\|_{\text{Lip}(B_X)}\leq\|P\|_v \cdot \left( 2(1-s(1/2,N)^2)^{-1}\sum\limits_{n=1}^N\frac{2^{2n-1}n^n}{n!} \right)$$
as desired.
\end{proof}

\begin{lemma}\label{convlemma}
Let $X$ be a Banach space, $\ep>0$, and $(\rho_n)_{n \in \N} \subset \R^+$ be an absolutely summable sequence. If for every $n\in \N$ we have that
\begin{equation}\label{hypoR}d(x_{n+1},\text{span}_\C(x_n))\leq\rho_n,\end{equation}
then there is $x\in \overline{B_X}$ and a subsequence $(x_{\sigma(n)})_{n \in \N}$ such that $\|x_{\sigma(n)}-x\|\to0$.
\end{lemma}
\begin{proof}
If there is a subsequence of $(x_n)$ converging to 0 we are done. Otherwise, we may assume that there is $r>0$ such that $\|x_n\|\ge r$ for every $n\in\N$. If we denote $Y_n=\text{span}_\C(x_n)$, we claim that
\begin{equation}\label{claimR}\sup\limits_{x\in S_{Y_{n}}}d(x,S_{Y_{n+1}})\leq\frac{2\rho_n}{r}\;\;\;\;\forall n\in\N.\end{equation}
Indeed, from \eqref{hypoR} we know that there is $k\in Y_n$ such that
\begin{equation*} 
\left\|\frac{x_{n+1}}{\|x_{n+1}\|}-k\right\|\leq \frac{\rho_n}{\|x_{n+1}\|}\leq\frac{\rho_n}{r}.
\end{equation*} 
Hence, $|1-\|k\||\leq\rho_n/r$ and so
\begin{eqnarray*} 
d(S_{Y_{n+1}},S_{Y_n}) \leq \left\|\frac{x_{n+1}}{\|x_{n+1}\|}-\frac{k}{\|k\|}\right\| &\leq& \left\|\frac{x_{n+1}}{\|x_{n+1}\|}-k \right\|+ \left\|k-\frac{k}{\|k\|}\right\|\\ \\
&\leq& \frac{\rho_n}{r} + |1-\|k\|| \\
&\leq& \frac{2\rho_n}{r}.
\end{eqnarray*}

Therefore, there is $\widetilde k_{n+1}\in S_{Y_{n+1}}$ and $\widetilde k_n\in S_{Y_n}$ such that $\|\widetilde k_n-\widetilde k_{n+1}\|\leq\frac{2\rho_n}{r}$. Let us take an arbitrary $x\in S_{Y_n}$. Then, there is $\lambda\in\T$ such that $x=\lambda \widetilde k_{n}$. Hence, we finish the proof of \eqref{claimR} since
$$d(x,S_{Y_{n+1}})\leq\|x-\lambda \widetilde k_{n+1}\|=|\lambda|\|\widetilde k_n-\widetilde k_{n+1}\|\leq\frac{2\rho_n}{r}.$$
Now that \eqref{claimR} has been proven, let us tackle the proof of the lemma. We take here $k_1=\frac{x_1}{\|x_1\|}$ and inductively choose $k_{n+1}$ for every $n\in\N$ as the element of $S_{Y_{n+1}}$ satisfying that $\|k_{n+1}-k_n\|\leq\frac{2\rho_n}{r}$. Clearly, for every $n,m\in\N$, we have that 
$$\|k_{n+m}-k_n\|\leq\frac{2}{r}\sum\limits_{i\ge n}\rho_i.$$
Hence, the sequence $(k_n)_{n \in \N}$ is a Cauchy sequence defined on the sphere of the complete space $X$. Therefore, there is $k\in S_X$ such that $\|k_n-k\|\to0$. Finally, for every $n\in\N$ there must be $\lambda_n\in\overline{\D}$ such that $x_n=\lambda_n k_n$. Since $\overline{\D}$ is compact there is $\lambda\in \overline{\D}$ and $\sigma:\N\to\N$ strictly increasing such that $|\lambda_{\sigma(n)}-\lambda|\to0$. Thus, we are done taking $x=\lambda k$ because
$$\|x_{\sigma(n)}-x\|=\|\lambda_{\sigma(n)}k_{\sigma(n)}-\lambda k\|\leq|\lambda_{\sigma(n)}-\lambda|+\|k_{\sigma(n)}-k\|\to0.$$
\end{proof}

Let us recall that the {\it modulus of convexity}  of a Banach space $(X,\|\cdot\|)$ is the function $\delta:[0,2]\to[0,1]$ given by
\begin{equation*} 
\delta(t):=\inf\Big\{1-\Big\|\frac{x+y}{2}\Big\|\;:\;x,y\in B_X,\;\|x-y\|\ge t\Big\}.
\end{equation*} 

We will also need the following straightforward result.

\begin{fact}
Let $X$ be a Banach space.
Then for every $x\in X$ there is a norm one projection $P_x$ from $X$ onto $\text{span}_\C(x)$.
\end{fact}

For readability, we fix some notation we will be using for the rest of the proof. 

\noindent
\textbf{Notation}: Given some $x\in X$ and $\rho\in[0,1]$, we will denote
\begin{equation} \label{distance-mu}
[x]_\rho :=\{y\in X\;:\;d(y,\text{span}_\C(x))\leq\rho\}\;\;\text{ and }\;\;\mu(\rho):=\frac{\rho^2\delta(2\rho^2)^2}{16}.
\end{equation}
We will also need to consider the linear operator $T_{\rho,x}:X\to X$ given by
\begin{equation} \label{T-pho-x}
T_{\rho,x}(y)=(1-\rho)y+\rho P_x(y)\;\;\;\;(y\in X).
\end{equation} 

\vspace{0.5cm}

Let us notice that whenever $X$ is a uniformly convex Banach space, the function $\mu$ defined above satisfies the following easy-to-check properties:
\begin{itemize}
\item $\mu$ is strictly increasing.
\item $\mu(\rho)=0$ if and only if $\rho=0$.
\item $\mu(\rho)<\rho$ for every $\rho\in(0,1]$.
\end{itemize}

Now we move one step further in the proof of our main result. 

\begin{lemma}\label{Tprop}
Let $X$ be a uniformly convex Banach space, $Y$ an arbitrary Banach space, and $N\in\N$. Given $x\in B_X$, $0<\rho<1$, and $P\in\mathcal{P}^N(X,Y)$, the next properties are satisfied:
\begin{enumerate}
\item $\|P-P\circ T_{\rho,x}\|_\infty\leq\rho\|P\|_vM_N/2.$\label{T1}
\item $\|T_{\rho,x}(y)\|\leq\|y\|-\|y\|\delta(2\rho\|y-P_x(y)\|)$ for every $y\in B_X$.\label{T2}
\end{enumerate}
\end{lemma}
\begin{proof}
Let us prove \eqref{T1}. Given $y\in B_X$, by using Proposition \ref{lippol}, we have that
\begin{eqnarray*} 
\|P(y)-P\circ T_{\rho,x}(y)\| &=& \|P(y)-P\big((1-\rho)y+\rho P_x(y)\big)\| \\
&\leq& \|P\|_{\text{Lip}(B_X)}\rho\|y-P_x(y)\| \\
&\leq& \rho\|P\|_vM_N/2
\end{eqnarray*} 
Now, let us tackle the proof of \eqref{T2}. It is clear that $\|2\rho P_x(y)+(1-2\rho )y\|\leq \|y\|$. If we denote $w=2\rho P_x(y)+(1-2\rho )y$ and $z=T_{\rho,x}(y)$ then we have that $z=\frac{y+w}{2}$. Hence, we deduce (\eqref{T2}) since
\begin{eqnarray*} 
\left\| \frac{z}{\|y\|} \right\| = \left\|\frac{y/\|y\|+w/\|y\|}{2}\right\| &\leq& 1-\delta\left(\frac{\|y-w\|}{\|y\|}\right) \\
&\leq& 1-\delta\big(\|y-w\|\big)\\
&=& 1-\delta\big(2\rho\|y-P_x(y)\|\big).
\end{eqnarray*}
\end{proof}

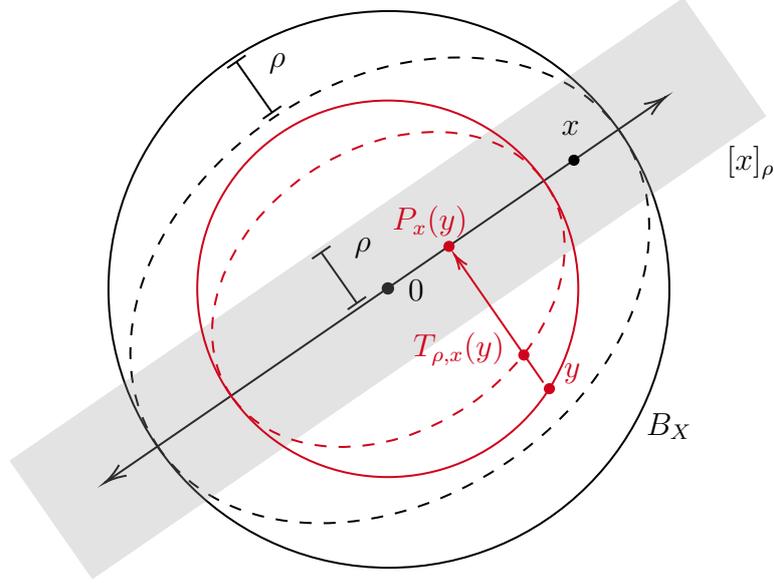
\begin{figure}
\centering

\tikzset{every picture/.style={line width=0.75pt}} 

\begin{tikzpicture}[x=0.75pt,y=0.75pt,yscale=-1,xscale=1]

\draw   (191,201) .. controls (191,123.68) and (253.68,61) .. (331,61) .. controls (408.32,61) and (471,123.68) .. (471,201) .. controls (471,278.32) and (408.32,341) .. (331,341) .. controls (253.68,341) and (191,278.32) .. (191,201) -- cycle ;
\draw  [fill={rgb, 255:red, 0; green, 0; blue, 0 }  ,fill opacity=1 ] (327.9,200.5) .. controls (327.9,199.06) and (329.06,197.9) .. (330.5,197.9) .. controls (331.94,197.9) and (333.1,199.06) .. (333.1,200.5) .. controls (333.1,201.94) and (331.94,203.1) .. (330.5,203.1) .. controls (329.06,203.1) and (327.9,201.94) .. (327.9,200.5) -- cycle ;
\draw  [color={rgb, 255:red, 0; green, 0; blue, 0 }  ,draw opacity=1 ][dash pattern={on 4.5pt off 4.5pt}][line width=0.75]  (216.44,281.14) .. controls (183.59,233.71) and (208.45,159.6) .. (271.96,115.6) .. controls (335.48,71.6) and (413.6,74.38) .. (446.45,121.81) .. controls (479.31,169.24) and (454.45,243.35) .. (390.93,287.35) .. controls (327.42,331.35) and (249.3,328.57) .. (216.44,281.14) -- cycle ;
\draw    (254.9,86.5) -- (272.68,112) ;
\draw [shift={(272.68,112)}, rotate = 235.11] [color={rgb, 255:red, 0; green, 0; blue, 0 }  ][line width=0.75]    (0,5.59) -- (0,-5.59)   ;
\draw [shift={(254.9,86.5)}, rotate = 235.11] [color={rgb, 255:red, 0; green, 0; blue, 0 }  ][line width=0.75]    (0,5.59) -- (0,-5.59)   ;
\draw    (467.56,105.34) -- (190.84,297.06) ;
\draw [shift={(189.2,298.2)}, rotate = 325.28] [color={rgb, 255:red, 0; green, 0; blue, 0 }  ][line width=0.75]    (10.93,-3.29) .. controls (6.95,-1.4) and (3.31,-0.3) .. (0,0) .. controls (3.31,0.3) and (6.95,1.4) .. (10.93,3.29)   ;
\draw [shift={(469.2,104.2)}, rotate = 145.28] [color={rgb, 255:red, 0; green, 0; blue, 0 }  ][line width=0.75]    (10.93,-3.29) .. controls (6.95,-1.4) and (3.31,-0.3) .. (0,0) .. controls (3.31,0.3) and (6.95,1.4) .. (10.93,3.29)   ;
\draw    (297.22,182.4) -- (315,207.9) ;
\draw [shift={(315,207.9)}, rotate = 235.11] [color={rgb, 255:red, 0; green, 0; blue, 0 }  ][line width=0.75]    (0,5.59) -- (0,-5.59)   ;
\draw [shift={(297.22,182.4)}, rotate = 235.11] [color={rgb, 255:red, 0; green, 0; blue, 0 }  ][line width=0.75]    (0,5.59) -- (0,-5.59)   ;
\draw  [draw opacity=0][fill={rgb, 255:red, 155; green, 155; blue, 155 }  ,fill opacity=0.28 ] (141.61,287.15) -- (478.07,54.18) -- (519.39,113.85) -- (182.93,346.82) -- cycle ;
\draw  [color={rgb, 255:red, 208; green, 2; blue, 27 }  ,draw opacity=1 ][fill={rgb, 255:red, 208; green, 2; blue, 27 }  ,fill opacity=1 ] (408.82,250.94) .. controls (408.82,249.68) and (409.84,248.65) .. (411.11,248.65) .. controls (412.37,248.65) and (413.4,249.68) .. (413.4,250.94) .. controls (413.4,252.21) and (412.37,253.23) .. (411.11,253.23) .. controls (409.84,253.23) and (408.82,252.21) .. (408.82,250.94) -- cycle ;
\draw [color={rgb, 255:red, 208; green, 2; blue, 27 }  ,draw opacity=1 ]   (408.07,247.87) -- (364.87,185.18) ;
\draw [shift={(363.73,183.53)}, rotate = 55.43] [color={rgb, 255:red, 208; green, 2; blue, 27 }  ,draw opacity=1 ][line width=0.75]    (7.65,-2.3) .. controls (4.86,-0.97) and (2.31,-0.21) .. (0,0) .. controls (2.31,0.21) and (4.86,0.98) .. (7.65,2.3)   ;
\draw  [color={rgb, 255:red, 208; green, 2; blue, 27 }  ,draw opacity=1 ] (235.33,200.7) .. controls (235.33,148.4) and (277.91,106) .. (330.43,106) .. controls (382.96,106) and (425.53,148.4) .. (425.53,200.7) .. controls (425.53,253) and (382.96,295.4) .. (330.43,295.4) .. controls (277.91,295.4) and (235.33,253) .. (235.33,200.7) -- cycle ;
\draw  [color={rgb, 255:red, 208; green, 2; blue, 27 }  ,draw opacity=1 ][dash pattern={on 4.5pt off 4.5pt}][line width=0.75]  (252.61,254.9) .. controls (230.3,222.82) and (247.19,172.7) .. (290.34,142.94) .. controls (333.49,113.18) and (386.56,115.07) .. (408.87,147.14) .. controls (431.18,179.22) and (414.29,229.35) .. (371.14,259.1) .. controls (327.98,288.86) and (274.92,286.98) .. (252.61,254.9) -- cycle ;
\draw  [color={rgb, 255:red, 208; green, 2; blue, 27 }  ,draw opacity=1 ][fill={rgb, 255:red, 208; green, 2; blue, 27 }  ,fill opacity=1 ] (358.72,179.35) .. controls (358.72,178.08) and (359.75,177.06) .. (361.02,177.06) .. controls (362.28,177.06) and (363.31,178.08) .. (363.31,179.35) .. controls (363.31,180.62) and (362.28,181.64) .. (361.02,181.64) .. controls (359.75,181.64) and (358.72,180.62) .. (358.72,179.35) -- cycle ;
\draw  [color={rgb, 255:red, 208; green, 2; blue, 27 }  ,draw opacity=1 ][fill={rgb, 255:red, 208; green, 2; blue, 27 }  ,fill opacity=1 ] (396.06,234.02) .. controls (396.06,232.75) and (397.08,231.72) .. (398.35,231.72) .. controls (399.62,231.72) and (400.64,232.75) .. (400.64,234.02) .. controls (400.64,235.28) and (399.62,236.31) .. (398.35,236.31) .. controls (397.08,236.31) and (396.06,235.28) .. (396.06,234.02) -- cycle ;
\draw  [color={rgb, 255:red, 0; green, 0; blue, 0 }  ,draw opacity=1 ][fill={rgb, 255:red, 0; green, 0; blue, 0 }  ,fill opacity=1 ] (421.06,136.02) .. controls (421.06,134.75) and (422.08,133.72) .. (423.35,133.72) .. controls (424.62,133.72) and (425.64,134.75) .. (425.64,136.02) .. controls (425.64,137.28) and (424.62,138.31) .. (423.35,138.31) .. controls (422.08,138.31) and (421.06,137.28) .. (421.06,136.02) -- cycle ;

\draw (269.76,80.48) node [anchor=north west][inner sep=0.75pt]    {$\rho $};
\draw (458,261.6) node [anchor=north west][inner sep=0.75pt]    {$B_{X}$};
\draw (312.43,173.7) node [anchor=north west][inner sep=0.75pt]    {$\rho $};
\draw (416.87,236.97) node [anchor=north west][inner sep=0.75pt]  [color={rgb, 255:red, 208; green, 2; blue, 27 }  ,opacity=1 ]  {$y$};
\draw (341.7,221.3) node [anchor=north west][inner sep=0.75pt]  [color={rgb, 255:red, 208; green, 2; blue, 27 }  ,opacity=1 ]  {$T_{\rho ,x}( y)$};
\draw (331.37,157.3) node [anchor=north west][inner sep=0.75pt]  [color={rgb, 255:red, 208; green, 2; blue, 27 }  ,opacity=1 ]  {$P_{x}( y)$};
\draw (339.1,194.5) node [anchor=north west][inner sep=0.75pt]    {$0$};
\draw (415.73,114.87) node [anchor=north west][inner sep=0.75pt]    {$x$};
\draw (498,127.4) node [anchor=north west][inner sep=0.75pt]    {$[ x]_{\rho }$};

\end{tikzpicture}
\caption{Action of $P_x$ and $T_{\rho,x}$ over an element $y\in B_X\setminus[x]_\rho$ for $\rho>0$ and $x\in B_X$.}\label{elipses}
\end{figure}

The next result contains most of the information of the method used to prove Theorem \ref{mainthdensity}. We invite the reader to check Figure \ref{elipses} for a conceptual picture with the different elements of the next Lemma \ref{estimate}.

\begin{lemma}\label{estimate}
Let $X$ be a uniformly convex Banach space, $Y$ an arbitrary Banach space, $N\in\N$, and $0<\rho<\frac{1}{16M_N}$. If $R\in\mathcal{P}^N(X;Y)$ with $1/2\leq\|R\|_v\leq2$ and $x\in B_X$  are such that
\begin{equation*} 
(1-\|x\|^2)\|R(x)\|\geq\|R\|_v-\mu(\rho),
\end{equation*} 
then, for every $S\in\mathcal{P}^N(X;Y)$ satisfying $\|S-R\circ T_{\rho,x}\|_\infty\leq\mu(\rho)$, we have that 
\begin{equation*} 
\sup\limits_{y\in B_X\setminus [x]_\rho}(1-\|y\|^2)\|S(y)\|<\|S\|_v-\mu(\rho).
\end{equation*}
\end{lemma}
As the reader may have realized, Lemma \ref{estimate} bounds the region where $P\circ T_{\rho,x}$ almost attains the $v$-norm. 
\begin{proof}[Proof of Lemma \ref{estimate}]
Let us argue by contradiction assuming that there is $y\in B_X\setminus[x]_\rho$ such that
\begin{equation}\label{contra}(1-\|y\|^2)\|S(y)\|\ge\|S\|_v-\mu(\rho).\end{equation}
Since $y\notin[x]_\rho$ we have that $\|y-P_x(y)\|>\rho$ and $\|y\|>\rho$. Thus, from property \eqref{T2} of Lemma \ref{Tprop} we deduce that
\begin{equation}\label{normT2}\|y\|\ge\|T_{\rho,x}(y)\|+\rho\delta(2\rho^2).\end{equation}
Now, from the fact that $T_{\rho,x}(x)=x$ and $\|S-R\circ T_{\rho,x}\|\leq\mu(\rho)$ it follows that
$$\|S\|_v\ge(1-\|x\|^2)\|S(x)\|\ge(1-\|x\|^2)\big(\|R\circ T_{\rho,x}(x)\|-\mu(\rho)\big)>\|R\|_v-2\mu(\rho).$$
Equivalently, we have that $\|R\|_v-4\mu(\rho)<\|S\|_v-2\mu(\rho)$. Hence, using our initial assumption \eqref{contra} we get that
$$\|R\|_v-4\mu(\rho)<(1-\|y\|^2)\|S(y)\|-\mu(\rho)\leq(1-\|y\|^2)\|R\circ T_{\rho,x}(y)\|.$$
Now, from \eqref{normT2} we obtain that
$$\begin{aligned}\|R\|_v-4\mu(\rho)<&\big(1-\|T_{\rho,x}(y)\|^2-\rho^2\delta(2\rho^2)^2\big)\|R(T_{\rho,x}(y))\|\\\leq&\|R\|_v-\rho^2\delta(2\rho^2)^2\|R(T_{\rho,x}(y))\|.\end{aligned}$$
Equivalently,
\begin{equation}\label{eqcontra}\frac{\rho^2\delta(2\rho^2)^2\|R(T_{\rho,x}(y))\|}{4}<\mu(\rho).\end{equation}
Using Property \eqref{T1} of Lemma \ref{Tprop} we also obtain
\begin{equation}\label{mid}\|R-S\|_v\leq\|R-S\|_\infty\leq \|R-R\circ T_{\rho,x}\|_\infty+\|R\circ T_{\rho,x}-S\|_\infty\leq\rho M_N+\mu(\rho).\end{equation}
It is straightforward to check that if $\rho<\frac{1}{16M_N}$ then $1/2-3\mu(\rho)-\rho M_N\ge1/4$. Hence, using \eqref{contra} and \eqref{mid} we have that
$$\begin{aligned}\|R(T_{\rho,x}(y))\|\ge&\|S(y)\|-\mu(\rho)\ge\frac{\|S\|_v-\mu(\rho)}{1-\|y\|^2}-\mu(\rho)>\|S\|_v-2\mu(\rho)\\\ge&\|R\|_v-3\mu(\rho)-\rho M_N\ge 1/2-3\mu(\rho)-\rho M_N\ge1/4.\end{aligned}$$
Therefore, by this last inequality and \eqref{eqcontra} we get that
\begin{equation*} 
\frac{\rho^2\delta(2\rho^2)^2}{16}<\mu(\rho),
\end{equation*} 
which is a contradiction.
\end{proof}

We are finally ready to provide the proof of Theorem \ref{mainthdensity}. We invite the reader to check the notations (\ref{M-N}), (\ref{distance-mu}), and (\ref{T-pho-x}) which were all defined throughout this section.

\begin{proof}[Proof of Theorem \ref{mainthdensity}] Let $X$ be a uniformly convex Banach space. Take an arbitrary Banach space $Y$ and fix $N \in \N$. We may assume without loss of generality that $\ep<1/16$ and take $\eta(\ep):=\mu\big(\frac{\ep}{2M_N}\big)$. Let us define $(\rho_n)_{n \in \N} \subseteq \R^+$ as follows. We set $\rho_1 :=\frac{\ep}{2M_N}$ and, for every $n\ge2$, we take the element $\rho_n :=\frac{\mu(\rho_1)}{{2^n}{M_N}}\in\R^+$. Therefore, the following holds,
\begin{equation}\label{ineqsums}
    M_N\sum\limits_{n\ge1}\rho_n\leq \ep \ \ \ \mbox{and} \ \ \ M_N\sum\limits_{n\ge2}\rho_n\leq\mu(\rho_1).
\end{equation}
We start now with the construction of $Q$. Let us define inductively a sequence of polynomials $(P_n)_{n \in \N}$ and a sequence of points $(x_n)_{n \in \N}$. We consider $P_1:=P \in \mathcal{P}(^N X; Y)$ and $x_1:=x\in B_X$ so that by hypothesis we have 
\begin{equation*} 
(1-\|x_1\|^2)\|P_1(x_1)\|>\|P_1\|_v-\mu(\rho_1). 
\end{equation*} 
Now, if $P_n$ and $x_n$ have been defined, we take $P_{n+1} :=P_n\circ T_{\rho_n,x_n}$ and $x_{n+1}\in B_X$ to be such that 
\begin{equation*} 
(1-\|x_{n+1}\|^2)\|P_{n+1}(x_{n+1})\|>\|P_{n+1}\|_v-\mu(\rho_{n+1}).
\end{equation*} 
Since $M_N\sum\limits_{n=1}^\infty\rho_n\leq\ep$, we get that
\begin{equation}\label{ineqbas}\|P_n-P\|_\infty\leq\sum\limits_{i=1}^{n-1}\|P_i-P_{i+1}\|_{\infty}\leq\sum\limits_{i=1}^{n-1}M_N\rho_i=M_N\sum\limits_{n=1}^\infty\rho_n\leq\ep.\end{equation}
Therefore, since $\ep<1/16$ we have that
\begin{equation}\label{Pbound}\begin{aligned}\|P_n\|_v&\leq\|P\|_v+\|P_n-P\|_v\leq\|P\|_v+\|P_n-P\|_\infty<1+1/16<2,\\
\|P_n\|_v&\ge\|P\|_v-\|P_n-P\|_v\ge\|P\|_v-\|P_n-P\|_\infty>1-1/16>1/2.\end{aligned}\end{equation}
Since $\sum\limits_{n=1}^\infty\rho_n<\infty$, it is straightforward to see that $(P_n)_{n \in \N}$ is a Cauchy sequence. In fact, if $m,n\in\N$, then by property \eqref{T1} of Lemma \ref{Tprop} and \eqref{Pbound}, we have that 
\begin{eqnarray*}
\|P_n-P_{n+m}\|_\infty &\leq& 
\sum\limits_{i=n}^{n+m-1}\|P_i-P_{i+1}\|_\infty \\
&=& \sum\limits_{i=n}^{n+m-1}\|P_i-P_{i}\circ T_{\rho_i,x_i}\|_\infty \\
&\leq& \sum\limits_{i\ge n}\rho_iM_N.
\end{eqnarray*} 
Therefore, by completeness, there is $Q\in \mathcal{P}^N(X;\C)$ such that $\|Q-P_n\|_\infty\to0$.

\vspace{0.2cm}

Now, from \eqref{ineqbas}, it follows that $\|Q-P\|_\infty=\lim_n\|P_n-P\|_\infty\leq\ep$. It only remains to prove that there is $y\in B_X\cap[x]_\ep$ such that $(1-\|y\|^2)|Q(y)|=\|Q\|_v$. For that purpose, we are going to prove that a subsequence of $(x_n)_{n \in \N}$ converges (in norm) to some $y\in B_X$. Indeed, let us fix $n\in\N$ and check that the conditions of Lemma \ref{estimate} are satisfied with $R=P_n$, $S=P_{n+1}$, $x=x_n$, and $\rho=\rho_n$:

Clearly, $0<\rho_n<\frac{1}{16M_N}$. By \eqref{Pbound}, we also know that $1/2<\|P_n\|_v<2$. It is also immediate that $\|P_{n+1}-P_n\circ T_{\rho_n,x_n}\|_\infty=0<\mu(\rho_n)$. Finally, from the definition of the sequence $(x_n)_{n \in \N}$ we check the last condition of Lemma \ref{estimate}, namely,
$$(1-\|x_n\|^2)\|P_n(x_n)\|>\|P_n\|_v-\mu(\rho_n).$$
Therefore, by Lemma \ref{estimate}, we have that
\begin{equation}\label{neigh}\sup\limits_{y\in B_X\setminus[x_n]_{\rho_n}}(1-\|y\|^2)\|P_{n+1}(y)\|<\|P_{n+1}\|_v-\mu(\rho_n).\end{equation}
Since $\mu$ is increasing, we get that
$$(1-\|x_{n+1}\|^2)\|P_{n+1}(x_{n+1})\|>\|P_{n+1}\|_v-\mu(\rho_{n+1})\ge\|P_{n+1}\|_v-\mu(\rho_n).$$
Hence, by \eqref{neigh} $x_{n+1}\in B_X\cap[x_n]_{\rho_n}$, that is,
$$d(x_{n+1},\text{span}_\C(x_n))\leq\rho_n.$$
Thus, by Lemma \ref{convlemma}, there is a subsequence $(x_{\sigma(n)})_{n \in \N}$ and a point $y\in \overline{B_X}$ such that $\|x_{\sigma(n)}-y\|\to0$. Notice that $Q$ attains its weighted norm in $y$ since
\begin{eqnarray*} 
(1-\|y\|^2)|Q(y)| &=& \lim\limits_{n\to\infty}(1-\|x_{\sigma(n)}\|^2)\|P_{\sigma(n)}(x_{\sigma(n)})\| \\
&\ge& \lim\limits_{n\to\infty}\|P_{\sigma(n)}\|_v-\mu(\rho_{\sigma(n)})\\
&=&\|Q\|_v.
\end{eqnarray*} 
It only remains to show that $y\in [x]_\ep$. Indeed, using the same argument as in \eqref{ineqbas}, and taking into account \eqref{ineqsums} we deduce that
$$\|Q-P\circ T_{\rho_1,x_1}\|_\infty\leq \mu(\rho_1).$$
Therefore, by using Lemma \ref{estimate} with $R=P$, $S=Q$, $x=x_1$ and $\rho=\rho_1$ we get that
$$\sup\limits_{z\in B_X\setminus [x_1]_{\rho_1}}(1-\|z\|^2)\|Q(z)\|<\|Q\|_v-\mu(\rho_1).$$
Thus, since $Q$ attains the norm $\|Q\|_v$ in $y$, we finally conclude that $y\in[x_1]_{\rho_1}\subseteq [x]_\ep$ and we are done with the proof.
\end{proof}

\subsection{More examples} We conclude the paper by providing  examples of density of certain holomorphic functions which attain their $s$-norms. For a Banach space $X$, we denote as $\mathcal{A}_{wu}(B_X)$ the subspace of $\mathcal{A}_u(B_X;\C)$ consisting of holomorphic functions which are uniformly weakly continuous in the closed unit ball.

The following result arises as a direct consequence of the proof of \cite[Theorem 3.3]{AAGM}.

\begin{theorem}\label{finiteth}
Let $X$ be a Banach space satisfying the following property: for every finite-dimensional space $F$, $\ep>0$, and bounded linear operator $T:X\to F$, there is a norm one projection $P:X\to X$ with finite-dimensional range such that $\|T-TP\|\leq\ep$. Then, the subset of holomorphic functions in $\mathcal{A}_{wu}(B_X)$ attaining their norm $\|\cdot\|_s$ for every $s\in (0,1]$ is $\|\cdot\|_{\infty}$ dense in $\mathcal{A}_{wu}(B_X)$.
\end{theorem}

The next corollary also follows from \cite[Corollary 3.4]{AAGM}.

\begin{corollary}\label{corolfinite}
Let $X$ be a Banach space satisfying at least one of the following conditions:
\begin{itemize}
    \item It has a shrinking monotone finite-dimensional decomposition.
    \item $X=C(K)$ for a compact Hausdorff topological space $K$.
    \item $X=L_p(\mu)$ where $\mu$ is a finite measure and $p\in[0,\infty]$.
\end{itemize}
Then, the subset of holomorphic functions in $\mathcal{A}_{wu}(B_X)$ attaining the norm $\|\cdot\|_s$ for every $s\in (0,1]$ is $\|\cdot\|_{\infty}$ dense in $\mathcal{A}_{wu}(B_X)$. In particular, if $K$ is a scattered compact topological space, then the subset of holomorphic functions in $\mathcal{A}_{u}(B_{C(K)};\C)$ attaining the norm $\|\cdot\|_s$ for every $s\in (0,1]$ is $\|\cdot\|_{\infty}$ dense in the  $\mathcal{A}_{u}(B_{C(K)};\C)$.
\end{corollary}

\noindent 
\textbf{Acknowledgements}: The authors would like to thank Mingu Jung, who was always available, for several fruitful conversations on the topic of this manuscript. They are also thankful to María José Beltrán Meneu, José Bonet, Mario P. Maletzki, and Alejandro Miralles for replying some enquiries during the procedure of writing the paper. 

\noindent 
\textbf{Funding information}: S.~Dantas was supported by the Spanish AEI Project PID2019 - 106529GB - I00 / AEI / 10.13039/501100011033 and the funding received from the Universitat Jaume I through its Research Stay Grants (E-2022-04). R. Medina was supported by CAAS CZ.02.1.01/0.0/0.0/16-019/0000778, project
SGS21/056/OHK3/1T/13, MICINN (Spain) Project PGC2018-093794-B-I00, and MIU
(Spain) FPU19/04085 Grant.

\end{document}